\documentclass[12pt]{amsart}
\usepackage{comandi}
\newcommand{\mean}[1]{\left\langle #1 \right\rangle}
\newcommand{\bias}[1]{\overline{#1}}
\usepackage{amssymb,amsmath}
\usepackage{mathrsfs}
\usepackage{graphicx}
\def \pt {\partial_{t}}

\def \u{{\mathbf u}}
\def \g{{\mathbf g}}

\newfont{\eufmtwelve}   {eufm10 scaled \magstep1}
\newfont{\eufmten}      {eufm10 }
\newfont{\eufmnine}     {eufm9 }
\newfont{\eufmeight}    {eufm8 }
\newfont{\eufmseven}    {eufm7 }
\newfont{\eufmsix}      {eufm6 }
\newfont{\eufmfive}     {eufm5 }
\newfont{\eusmtwelve}   {eusm10 scaled \magstep1}
\newfont{\eusmten}      {eusm10}
\newfont{\eusmnine}     {eusm9 }
\newfont{\eusmeight}    {eusm8 }
\newfont{\eusmseven}    {eusm7 }
\newfont{\eusmsix}      {eusm6 }
\newfont{\eusmfive}     {eusm5 }
\newfont{\msbmtwelve}   {msbm10 scaled \magstep1}
\newfont{\msbmtw}   {msbm10 scaled \magstep3}
\newfont{\msbmeight}    {msbm8}
\newfont{\msbmten}      {msbm8}

\def \no#1#2#3 {{\bf #1} (#3), #2.}
\def \eds#1#2#3 {#1, #2, #3.}

\newcommand{\ddt}{\frac{\mbox{d}}{\mbox{dt}}}

\newcommand{\myH}[1]{\ensuremath{\mathcal{H}_{#1}}}
\newcommand{\dataConst}{A_{t,\tau}}

\newcounter{enumi_tmp}
\newcounter{enumi_saved}
\newenvironment{hypothesis}{\begin{enumerate}%
        \setcounter{enumi}{\value{enumi_saved}}%
    }{%
        \setcounter{enumi_saved}{\value{enumi}}%
        \end{enumerate}
    }

\newcounter{enumi_saved_bis}
\newenvironment{hypothesis_bis}{\begin{enumerate}%
        \setcounter{enumi}{\value{enumi_saved_bis}}%
    }{%
        \setcounter{enumi_saved_bis}{\value{enumi}}%
        \end{enumerate}
    }
\newcounter{enumi_saved_tris}
\newenvironment{hypothesis_tris}{\begin{enumerate}%
        \setcounter{enumi}{\value{enumi_saved_tris}}%
    }{%
        \setcounter{enumi_saved_tris}{\value{enumi}}%
        \end{enumerate}
    }

\def\ds{\displaystyle}

\def \l {\langle}
\def \r {\rangle}

\subjclass[2000]{35B41, 35B45, 35K25}


\begin{document}

\title[Pullback exponential attractor for $2D$ CHNS system]{Pullback exponential attractor for a Cahn-Hilliard-Navier-Stokes system in $2D$}

\author{Stefano Bosia}
    \address{Politecnico di Milano
        \newline\indent Dipartimento di Matematica ``F. Brioschi''
        \newline\indent Piazza Leonardo da Vinci 32
        \newline\indent I-20133 Milano, Italy}
    \email{stefano.bosia@mail.polimi.it}

\author{Stefania Gatti}
    \address{Universit\`a di Modena e Reggio Emilia
        \newline\indent Dipartimento di Scienze Fisiche, Informatiche e Matematiche
        \newline\indent Via Campi 213/B
        \newline\indent I-41125 Modena, Italy}
    \email{stefania.gatti@unimore.it}

\begin{abstract}
    We consider a model for the evolution of a mixture of two incompressible and partially immiscible Newtonian fluids in two dimensional bounded domain.
     More precisely, we address the well-known model $H$ consisting of the Navier-Stokes equation with non-autonomous external forcing term for the (average) fluid velocity, coupled with a convective Cahn-Hilliard equation with polynomial double-well potential describing the evolution of the relative density of atoms of one of the fluids. We study the long term behavior of solutions and prove that the system possesses a pullback exponential attractor. In particular the regularity estimates we obtain depend on the initial data only through fixed powers of their norms and these powers are uniform with respect to the growth of the polynomial potential considered in the Cahn-Hilliard equation.
\end{abstract}

\maketitle

\section{Introduction}

The modeling of multi-phase flows has been a field of intense mathematical research in the last years. In particular, its relevance for applications (see, e.g., \cite{Gurtin1996, Heida2012, Hohenberg1977, Morro2010} and references therein) has motivated several different approaches, among which diffuse interface methods seem to combine both numerical efficiency and theoretical tractability. One of their instances is given by the so called model $H$, which was first proposed in~\cite{Hohenberg1977} and~\cite{Siggia1979} and then rigorously derived in~\cite{Gurtin1996}. In the corresponding system, the two-phase flow is described by a (mean) velocity field~\vect{u}, which satisfies a Navier-Stokes type equation, and an order parameter field~$\psi$, which represents the difference of the relative concentrations of the two fluids and which solves a convective Cahn-Hilliard equation. More precisely, given a bounded and smooth domain $\Omega \subset \mathbb{R}^{2}$, assuming that the viscosity of the mixture is a constant $\nu>0$, the two-dimensional model $H$ reads as
\begin{equation}\label{E:PDE}
    \begin{cases}
        \dt{\vect{u}} + \vect{u} \cdot \nabla \vect{u} - \nu \lapl \vect{u}= \nabla p + \mu \nabla \psi + \vect{g},\quad \text{in $\Omega$,}\\
        \nabla \cdot \vect{u} = 0, \quad \text{in $\Omega$,}\\
        \dt{\psi} + (\vect{u} \cdot \nabla) \psi = m \lapl \mu, \quad \text{in $\Omega$,}\\
        \mu = -\epsilon\lapl \psi + \tfrac{1}{\epsilon} f(\psi), \quad \text{in $\Omega$,}\\
    \end{cases}
\end{equation}
where $\mu$ is the so called chemical potential with constant mobility $m \geqslant 0$, $\vect{g}$ is a time-dependent bulk force, $f$ is the derivative of a double-well potential~$F$ while $p = \ds \pi + \frac{\epsilon}{2} |\nabla \psi|^2 + \tfrac{1}{\epsilon} F(\psi)$ introduces the pressure~$\pi$ as well as the positive parameter~$\epsilon$ rendering the interaction between the two phases. In particular, $\epsilon$ is related with the small but not negligible thickness of the interface.

This system is usually complemented by homogeneous Dirichlet boundary conditions on the velocity field, no flux boundary conditions on the order parameter field and chemical potential, namely,
\begin{equation}\label{E:BC}
    \vect{u}    = \mathbf 0, \quad \dnu{\psi}  = 0, \quad \dnu{\mu} = 0, \qquad \text{on $\partial \Omega$}.
\end{equation}
Being the problem non-autonomous, we specify the initial values at a given time $\tau\in \mathbb R$ for the
state variables, that is,
\begin{equation*}
    \vect{u}(\tau) = \vect{u}_{0},\quad  \qquad \psi(\tau) = \psi_{0},\qquad \quad \text{in $ \Omega$}.
\end{equation*}
We recall that, in this model, the chemical potential of the binary mixture $\mu$ is given by the variational derivative of the free energy functional for the Cahn-Hilliard equation
\begin{equation*}
    \mathcal{F}(\psi) \eqdef \int_{\Omega} \left( \frac{\epsilon}{2} |{\nabla \psi}|^{2} + \eta F(\psi) \right)\, \mathrm{d}\vect{x},
\end{equation*}
where $F(\psi)$ is a suitable double-well potential characterizing the phase decomposition of the mixture. Since $F'=f$, the fourth equation in \eqref{E:PDE} follows.

The Cahn-Hilliard model for spinodal decomposition and coarsening during quenching of alloys was first proposed in~\cite{Cahn1958}. In this setting a thermodynamically consistent double-well potential $F$ is naturally seen to be logarithmic (see~\cite{Cahn1961} and references therein). However, this singular form for the potential causes major difficulties in the numerical and theoretical study of the Cahn-Hilliard system so that in applications it is often replaced by a polynomial approximation like
\begin{equation*}
    F(\psi) = \Const (1 - \psi^{2})^{2}.
\end{equation*}
In this context, a possible approach to deal with the physically relevant case consists in suitably approximating the singular potential by polynomials of increasing order (see~\cite{Frigeri2012a, Frigeri2012} for an application of this technique to a system closely related to the model $H$). Our paper deals with a polynomial potential $F$ of arbitrary order $p+3$, for $p\geqslant 1$ (in fact, the lower order case is much easier).

From the mathematical viewpoint, system~\eqref{E:PDE}-\eqref{E:BC} has been firstly studied in~\cite{Starovoitov1997} for $\Omega = \mathbb{R}^{2}$. Then, in the case of bounded domains, global existence results for both weak and strong solutions in the $2D$ case were obtained in~\cite{Boyer1999} (see also~\cite{Boyer2001}). More recently, the case of logarithmic potentials has been considered in~\cite{Abels2009a} (see also~\cite{Abels2009b}), where, in particular, the convergence of solutions to a single equilibrium has been established in absence of nongradient external forces. This issue has also been investigated in~\cite{Zhao2009} for smooth potentials. A rather complete picture of the longtime behavior in the case $N = 2$ on a bounded domain can be found in~\cite{Gal2010}. In the case $N = 3$, existence of trajectory attractors has been demonstrated in~\cite{Gal2010a} with time-dependent external forces. Many related models have also been extensively studied in recent years. Among others we recall non-Newtonian Ladyzhenskaya fluids (see~\cite{Grasselli2011} for the $2D$ case and~\cite{Bosia2013} for some partial results in the $3D$ case), non-local interactions (see~\cite{Frigeri2012a, Frigeri2012}) and chemically reacting mixtures (see~\cite{Bosia2013a} and references therein).

Concerning the mathematical theory of infinite dimensional dynamical systems, pullback exponential attractors represent a new instrument recently introduced in the literature (see~\cite{Langa2010}). The related theory combines both the advantages of pullback attractors in the non-autonomous case (see~\cite{Carvalho2013} for a comprehensive introduction to the theory of pullback attractors) and of the exponential attractors in Banach spaces in their most general form known today (see~\cite{Efendiev2000}).

The main result of this paper is the existence of a pullback exponential attractor for the system \eqref{E:PDE}-\eqref{E:BC}. As a byproduct, we derive several regularity estimates for the solutions to the Cahn-Hilliard-Navier-Stokes system. These have an interest of their own due to their uniform structure with respect to the growth of the double-well potential $F$. Indeed, if the potential $f(\psi)$ is assumed to satisfy $|f(\psi)| \leqslant \Const (|\psi|^{p+2} + 1)$, we are able to control the solutions only by suitable powers of the norms of the initial data independent of $p$. This is not an easy task since the computations repeatedly involve $f(\psi)$ and its derivatives, which are naturally estimated as
\begin{equation*}
    \Lpnorm[q]{f(\psi)} \leqslant \Const ( \Lpnorm[\infty]{\psi}^{p+2} +1)\leqslant \Const (\Lpnorm{\psi}^{\sfrac{(p+2)}{2}} \Lpnorm{\lapl \psi}^{\sfrac{(p+2)}{2}}+1),
\end{equation*}
carrying the polynomial character of $F$ directly into play. This obstacle is circumvented by suitably handling the nonlinear terms so that the dependence on the ``shape'' of the potential is confined to the multiplicative constants appearing in our results.

Therefore, these estimates can be seen as a preliminary step forwards an effective approximating procedure able to deal with the more physically relevant case given by the singular potential.

The plan of the paper goes as follows. In Section~\ref{S:functional setting} we introduce the functional setting required to study system~\eqref{E:PDE} and the main results obtained in this work. After recalling the theory of pullback exponential attractors in Section~\ref{S:theory}, we first derive basic energy estimates (Section \ref{S:Existence}) and then higher order regularity estimates (Section~\ref{S:higher_order}): in particular, our results are uniform w.r.t.\ the shape of $F$ in the sense made precise above. We then derive continuity results and time regularity for solutions in
Sections \ref{S:continuous dependence} and \ref{S:time regularity}. Finally, in Section~\ref{S:validation} we are able to check all the assumptions of the abstract results of Section~\ref{S:theory} in the case of system~\eqref{E:PDE}, concluding the proof of our main theorem.

\section{Functional setting and main results}\label{S:functional setting}

We will denote by $\Omega$ a smooth bounded domain of $\mathbb{R}^{2}$. The spaces $\Lp[p](\Omega)$ will be the usual Banach spaces of $p$-integrable functions with $p \in [1, \infty]$. The Sobolev-Hilbert space, which consists of $k$-differentiable functions  in the sense of distributions with square integrable derivatives, will be denoted by~$\Hs[k](\Omega)$. We shall use the bold symbols~$\Lpvect[p](\Omega)$ and $\Hsvect[k](\Omega)$ for the corresponding spaces of vector valued functions. The space of functions belonging to $\Hs[k](\Omega)$ and vanishing on the boundary will be denoted by $\Hs[k]_{0}(\Omega)$. Norms in the Sobolev spaces $\Hs[k](\Omega)$ will be denoted by $\norm{\Hs[k]}{\cdot}$, whereas we will use the shorthand notation $\Lpnorm[p]{\cdot}$ for the norm in $L^{p}(\Omega)$ spaces, $1 \leqslant p \leqslant \infty$. In order to study the velocity field $\vect{u}$ we introduce the usual framework of divergence-free distributions, i.e.,
\begin{equation*}
    \mathcal{V} \eqdef \{ \vect{\phi} \in \Cont[\infty]_{c}(\Omega; \mathbb{R}^{2}) \mid \nabla \cdot \vect{\phi} = 0 \}.
\end{equation*}
Then we consider its closure under suitable distributional norms
\begin{equation*}
    \Lpvect_{\divfree}(\Omega)          \eqdef \closure[\Lpvect(\Omega)]{\mathcal{V}}, \quad
    \Hsvect_{0, \divfree}(\Omega)       \eqdef \closure[\Hsvect_{0}(\Omega)]{\mathcal{V}}.
\end{equation*}
We also introduce the Leray projector $\mathbb{P} \colon \Lpvect(\Omega) \to \Lpvect_{\divfree}(\Omega)$ mapping every element of $\Lpvect(\Omega)$ to its divergence-free part. Furthermore, we will indicate by $\Hsvect[-1]_{\divfree}(\Omega)$ the dual space of $\Hsvect_{0, \divfree}(\Omega)$. In $\Hs_{0}(\Omega)$ and $\Hsvect_{0, \divfree}(\Omega)$ we will consider the following norms
\begin{equation*}
    \norm{\Hs_{0}(\Omega)}{f}^{2} \eqdef \Lpnorm{\nabla f}^{2} = \sum_{i = 1}^{n} \int_{\Omega} | f_{,i} |^{2} \, \mathrm{d}\vect{x} \qquad \norm{\Hsvect_{0, \divfree}(\Omega)}{\vect{f}}^{2} \eqdef \Lpnorm{\nabla \vect{f}}^{2} = \sum_{i,j = 1}^{n} \int_{\Omega} | \vect{f}_{i,j} |^{2} \, \mathrm{d}\vect{x}.
\end{equation*}
Finally we will denote by $\duality{f}{g}$ both the scalar product in $\Lp(\Omega)$ (or $\Lpvect(\Omega)$) and the duality pairing between $\Hs[-1](\Omega)$ and $\Hs_{0}(\Omega)$ (or their vector valued analogues), the exact meaning being clear from the context.

Since the second equation in \eqref{E:PDE} together with the boundary condition imply that the bulk integral of the order parameter is preserved by the evolution, we need to suitably account for this feature. First of all we define the mean value of $f$ over the domain $\Omega$ as
\begin{equation*}
    \mean{f} \eqdef \frac{1}{|\Omega|} \int_{\Omega} f \, \mathrm{d}\vect{x},
\end{equation*}
denoting by $\bias{f}$ the mean free part of $f$, that is,
\begin{equation*}
    \bias{f} \eqdef f - \mean{f}.
\end{equation*}
Thus, up to a shift of the order parameter field, we can always assume that the mean of $\psi$ is zero at the initial time and, due to the conservation of mass enforced by the Neumann boundary conditions, this will remain true for all positive times. Then the order parameter will belong to subspaces of $\Lp[p](\Omega)$ and $\Hs[k](\Omega)$ consisting of functions with zero mean, defined as
\begin{equation*}
    \Lp[p]_{(0)}(\Omega) \eqdef \{ v \in \Lp[p](\Omega) \mid \mean{v} = 0 \}, \quad \Hs[k]_{(c0)}(\Omega) \eqdef \{ v \in \Hs[k](\Omega) \mid \mean{v} = 0 \}.
\end{equation*}
Here we can use Poincar\'e's inequality (and some of its variants) at several stages when estimating the Sobolev norms of $\psi$. Indeed, the boundary conditions and the above definitions imply that
\begin{equation*}
    \int_{\Omega} \psi = 0, \quad \dnu{\psi} = 0 \text{ on $\partial \Omega$}, \quad \int_{\Omega} \lapl \psi = \int_{\partial \Omega} \dnu{\psi} = 0, \quad \dnu{\lapl \psi} = 0 \text{ on $\partial \Omega$}.
\end{equation*}
Therefore, all the norms $\norm{\Hs[j]}{\psi}$, $j = 1, \ldots, 4$ are equivalent to the $\Lp$-norms of the derivatives of order $j$. Moreover, Korn's inequality holds. Thus we have
\begin{equation*}
    \norm{\Hs}{\psi} \sim \Lpnorm{\nabla \psi}, \quad \norm{\Hs[2]}{\psi} \sim \Lpnorm{\lapl \psi}, \quad \norm{\Hs[3]}{\psi} \sim \Lpnorm{\nabla \lapl \psi}, \quad \norm{\Hs[4]}{\psi} \sim \Lpnorm{\lapl^{2} \psi}.
\end{equation*}

Finally, the functional spaces for the whole solution $(\vect u,\psi)$ are
\begin{equation*}
    \myH0 \eqdef \Lpvect_{\divfree}(\Omega) \times \Hs_{(c0)}(\Omega) \qquad \myH1 \eqdef \Hsvect_{0,\divfree}(\Omega) \times \Hs[2]_{(c0)}(\Omega),
\end{equation*}
which arise naturally in the study of the process generated by the solution of system~\eqref{E:PDE}.

We can now list the assumptions on the potential $F(\psi)$, starting  with some hypotheses concerning its regularity and growth:
\begin{hypothesis}
    \item \label{Hp:regularity} $F \in \Cont[5](\mathbb{R})$.
    \item \label{Hp:growth} $F(y)$ grows at most polynomially fast at infinity, namely
        \begin{equation*}
            |f''(y)| \leqslant C_f (1 + |y|^{p}),
        \end{equation*}
        for some positive constants $p$ and $C_f$.
    \item \label{Hp:coercivity} The potential is coercive, i.e.\ there exist positive real numbers $q$ and $c_f$ such that
        \begin{equation}\label{E:F_coercivity}
            F(y) \geqslant c_f \left( |y|^{2+q} - 1 \right)
        \end{equation}
        holds for all $y \in \mathbb{R}$.
\end{hypothesis}
Recalling that the potential $F$ appears in system~\eqref{E:PDE} only through its derivative, without loss of generality we can further assume that
\begin{hypothesis}
    \item \label{Hp:positivity} the functional $F$ is strictly positive, i.e.\ $F(y) > 0$, $\forall y \in \mathbb{R}$.
\end{hypothesis}

We now give some additional assumptions concerning the shape of the double-well potential~$F$.
\begin{hypothesis}
    \item \label{Hp:splitting} $F(y)$ is a quadratic perturbation of a regular convex function defined on the whole $\mathbb{R}$, that is,
        \begin{equation*}
            F(y) = F_{0}(y) - \alpha y^{2} + \gamma y + \beta,
        \end{equation*}
        where $F_{0} \in \Cont[5](\mathbb{R})$ is convex and $\alpha \in \mathbb{R}$ is a positive constant.
    \item \label{Hp:constants} Up to a suitable choice of the constants $\beta$ and $\gamma$ in Assumption~\eqref{Hp:splitting}, the convex part of the potential $F_{0}$ satisfies:
        \begin{equation*}
            F_{0}(0) = F'_{0}(0) = 0.
        \end{equation*}
\end{hypothesis}
In order to obtain higher order estimates having uniform structure with respect to the growth of $f$, we will assume that $F$ behaves as a polynomial at infinity. In particular we will suppose
\begin{hypothesis}
    \item \label{Hp:polynomial_growth} The relation $q = p+1$ holds in Assumptions~\eqref{Hp:growth} and~\eqref{Hp:coercivity}.
        Moreover, for any $k = 0,1,2,3,4$ there exists a positive constant $c_{k}$ such that
        \begin{equation} \label{Fcontrol}
            |f^{(k)}(y)|\leqslant c_{k} [1 + F(y)^{\sfrac{(p+2-k)}{(p+3)}}],\quad \forall y\in \mathbb{R}.
        \end{equation}
\end{hypothesis}

\begin{remark}\label{rem:fp+3}
    In the following Assumption~\eqref{Hp:polynomial_growth} will be necessary in order to estimate the $\Lp[r]$ norm of the derivatives of the potential $F$ in terms of some $\Lp[s]$ norm of the potential itself. In particular, when dealing with higher order estimates we will often use the immediate consequence of Assumption~\eqref{Hp:polynomial_growth}
    \begin{equation*}
        \Lpnorm[\frac{p+3}{p+2-k}]{f^{(k)}(\psi)}^{\sfrac{(p+3)}{(p+2-k)}} \leqslant \Const (1 + \Lpnorm[1]{F(\psi)}),
    \end{equation*}
    for any $\psi$ such that $F(\psi)\in \Lp[1](\Omega)$ and for some constant $\Const$ depending on $k$ and $\Omega$.
\end{remark}

\begin{remark}
Throughout the paper, we will always assume that $p \geqslant 2$, being much easier the case when the potential $F(y)$ grows at most as $y^{5}$ at infinity. In particular, all estimates in the following sections hold for $p \geqslant 1$, except for ~\eqref{E:continuous_dependence_last_term} below. However, note that a suitable estimate for this term can be produced also in the case $p \in [1,2)$ (cf.~Remark~\ref{R:continuous_dependence_p=1}) under the assumption
    \begin{equation*}
        f^{(iv)}(y) \leqslant \Const \qquad \forall y \in \mathbb{R}.
    \end{equation*}
    The case $p = 1$ is particularly relevant for applications since the polynomial potential $F(y) = (y^{2} - 1)^{2}$, which is often used in numerical simulations, falls in this setting.
\end{remark}

Finally, we assume the non-autonomous forcing term (symbol) $\vect{g}$ appearing in equation~\eqref{E:PDE} satisfies the following conditions:
\begin{hypothesis_tris}
    \item \label{Hp:gL2loc} $\g\in \Bochner{\Lp_{\loc}}{-\infty}{\infty}{\Lpvect_{\divfree}(\Omega)}$.
    \item \label{Hp:gM2} $\g\in \Bochner{\Lp_{\uloc}}{-\infty}{t}{\Lpvect_{\divfree}(\Omega)}$, for any $t\in \mathbb R$, that is,
        \begin{equation*}
            M_\g(t)\doteq \sup_{r\leqslant t} \int_{r-1} ^r \Lpnorm{\vect{g}(s)}^{2} \, \mathrm{d} s<\infty  ,\quad \forall t\in \mathbb{R}.
        \end{equation*}
    \item \label{Hp:gMq} There exist $t_0\in \mathbb{R}$ and $q>2$ such that $\g\in \Bochner{\Lp[q]_{\uloc}}{-\infty}{t_0}{\Lpvect_{\divfree}(\Omega)}$, namely,
        \begin{equation*}
            M_{\g,q}(t_0)\doteq \sup_{r\leqslant t_0} \int_{r-1} ^r \Lpnorm{\vect{g}(s)}^{q} \, \mathrm{d} s<\infty .
        \end{equation*}
\end{hypothesis_tris}

In this paper we will prove the following main results:
\begin{theorem}\label{T:main_result}
    Assume that $\vect{g}$ satisfies~\eqref{Hp:gL2loc} and~\eqref{Hp:gMq} and let $U_{\vect{g}}(t, \tau) \colon \myH0 \to \myH0$ be the solution operator for the system~\eqref{E:PDE}. Then there exists a family $\widetilde{\mathcal{M}}_{U_{\vect{g}}} = \{ \widetilde{\mathcal{M}}_{U_{\vect g}}(t) \colon t \leqslant t_{0} \}$ of nonempty compact subsets of $\myH1$, which is a pullback exponential attractor for system~\eqref{E:PDE} (see Theorem~\ref{T:continuous_time_attractor} below) in the topology of~$\myH1$.
\end{theorem}

\begin{corollary}\label{cor:process_result}
    Under the same assumptions of Theorem~\ref{T:main_result}, if, moreover, $\vect{g}$ satisfies~\eqref{Hp:gMq} uniformly for $t_{0} \in \mathbb{R}$ (as in~\eqref{Hp:gM2}), then the process $U_{\vect{g}}(t, \tau) \colon \myH1 \to \myH1$ has a family $\widetilde{\mathcal{M}}_{U_{\vect{g}}} = \{ \widetilde{\mathcal{M}}_{U_{\vect g}}(t) \colon t \in \mathbb R \}$ of nonempty compact subsets of $\myH1$, which is a pullback exponential attractor for system~\eqref{E:PDE} (see Theorem~\ref{T:process_attractor} below) in the topology of~$\myH1$.
\end{corollary}

\section{Exponential pullback attractors}\label{S:theory}

In this section we briefly review the theory of exponential pullback attractors as developed in~\cite{Langa2010}. Below, $(H, | \cdot |)$ and $(V, \norm{}{\cdot})$ will be two Banach spaces such that $V$ is compactly embedded in $H$. Both spaces have a metric structure therefore, given any two nonempty subsets $D_{1}$, $D_{2}$ of the metric space $X=H,V$, the Hausdorff semidistance and distance are well defined respectively as
\begin{equation*}
    \dst_{X}(D_{1}, D_{2}) \eqdef \sup_{v_{1} \in D_{1}} \inf_{v_{2} \in D_{2}} \norm{X}{v_{1} - v_{2}}
\end{equation*}
and
\begin{equation*}
    \dst^{\sym}_{X}(D_{1}, D_{2}) \eqdef \max \{ \dst_{X}(D_{1}, D_{2}), \dst_{X}(D_{2}, D_{1}) \}.
\end{equation*}

As in the usual case of exponential attractors~\cite{Efendiev2000} the key point of the argument is the introduction of a set of mappings which enjoy a suitable ``smoothing property''. This is responsible for the exponential convergence of the trajectories of the system to an exponentially attracting finite-dimensional compact set (an exponential attractor) as soon as the trajectories have entered a sufficiently small neighbourhood of the attractor itself. Let $\delta, K \in \mathbb{R}$ be positive constants and let $B$ be a bounded and closed subset of $V$. Define $\mathbb{S}_{\delta, K}(B)$ to be the class of mappings $S \colon V \to V$ such that the smoothing property holds on a $\delta$-neighbourhood (in $V$) of $B$, i.e.
\begin{equation*}
    S(\mathcal{O}_{\delta}(B)) \subset B
\end{equation*}
and
\begin{equation*}
    \norm{}{S v_{1} - S v_{2}} \leqslant K |v_{1} - v_{2}| \qquad \text{for all $v_{1}, v_{2} \in \mathcal{O}_{\delta}(B)$},
\end{equation*}
where $\mathcal{O}_{\delta}(B) \eqdef \{ v \in V \mid \inf_{w \in B} \norm{}{v-w} < \delta \}$ is a $\delta$-neighbourhood of the set $B$ in $V$.

We introduce a suitable class of family of mappings, which are the abstract, discrete-time, dynamical system representation of the evolution equations we will be interested in. In particular, let $n_{0} \in \mathbb{Z}$ be fixed and consider the class $\mathcal{U}_{d}(V, n_{0})$ of all families $U = \{ U(m,n) \mid n, m \in \mathbb{Z}, n \leqslant m \leqslant n_{0}\}$ of mappings $U(m,n) \colon V \to V$ such that
\begin{enumerate}
    \item $U(n,n) = Id$ for all $n \leqslant n_{0}$;
    \item $U(m,k) U(k,n) = U(m,n)$ for any $n \leqslant k \leqslant m \leqslant n_{0}$.
\end{enumerate}
When dealing with pullback attractors, only the evolution of the system up to the ``present'' time $n_{0}$ is of interest. The key question is how perturbations of the system in the past affect the present dynamic and what actually is the state of the system observed. This is the reason why elements belonging to the class $\mathcal{U}_{d}(V, n_{0})$ are defined up to time $n_{0}$ and not necessarily beyond.

We can say that element of the discrete time class $\mathcal{U}_{d}(V, \tau_{0})$ possess a discrete time pullback exponential attractor in the sense made precise by the following theorem
\begin{theorem}[{\cite[Theorem~2.1]{Langa2010}}]\label{T:discrete_time_attractor}
    Let $n_{0} \in \mathbb{Z}$, $\delta > 0$, $K > 0$ and $B \subset V$ be fixed with $B$ bounded and closed in $V$. Then, there exist positive constants $\Const_{1}$, $\Const_{2}$, $\overline{\epsilon}$ and $\alpha$ only depending on $V$, $H$, $\delta$, $K$ and $B$, such that, for each $U \in \mathcal{U}_{d}(V, n_{0})$ satisfying
    \begin{equation*}
        U(n, n-1) \in \mathbb{S}_{\delta, K}(B) \qquad \text{for all $n \leqslant n_{0}$},
    \end{equation*}
    there exists a family $\mathcal{M}_{U} = \{ \mathcal{M}_{U}(n) \mid n \leqslant n_{0} \}$, of nonempty subsets of $V$, which satisfies
    \begin{enumerate}[a)]
        \item $\mathcal{M}_{U}$ is positively invariant i.e.
            \begin{equation*}
                U(m,n) \mathcal{M}_{U}(n) \subset \mathcal{M}_{U}(m) \quad \text{for all $n \leqslant m \leqslant n_{0}$,}
            \end{equation*}
        \item $\mathcal{M}_{U}(n) \subset B$ is a compact subset of $V$, with finite fractal dimension estimated by
            \begin{equation*}
                \log_{2} N_{\epsilon}(\mathcal{M}_{U}(n), V) \leqslant \Const_{1} \log_{2} \frac{1}{\epsilon} + \Const_{2} \quad \text{for all $0 < \epsilon < \overline{\epsilon}$ and any $n \leqslant n_{0}$,}
            \end{equation*}
            where $N_{\epsilon}(\mathcal{M}_{U}(n), V)$ is the minimal number of $\epsilon$-balls in $V$, which are necessary to cover $\mathcal{M}_{U}(n)$,
        \item $\mathcal{M}_{U}$ attracts $B$ exponentially in a pullback sense i.e.
            \begin{equation*}
                \dst_{V}(U(m,n)B, \mathcal{M}_{U}(m)) \leqslant \Const_{1} e^{-\alpha (m-n)} \qquad \text{for all $n \leqslant m \leqslant n_{0}$,}
            \end{equation*}
        \item for every integer $k \leqslant 0$
            \begin{equation*}
                \mathcal{M}_{U}(n+k) = \mathcal{M}_{T_{k}U}(n) \qquad \text{for all $n \leqslant n_{0}$},
            \end{equation*}
            where $T_{k}U(m,n) \eqdef U(m+k, n+k)$.
    \end{enumerate}
\end{theorem}

\begin{remark}
    The results in~\cite{Langa2010} also include robustness of a discrete-time exponential pullback attractor w.r.t.\ to a suitable metric in the space of discrete time processes. For simplicity we do not mention all the pertinent details here. However, we recall that this result is important in deducing the analogue continuous-time theory and, in particular, in obtaining continuity in time of an exponential pullback attractor.
\end{remark}

Having in mind the more relevant continuous-time setting, we now introduce a suitable analogue of the class $\mathcal{U}_{d}(V, n_{0})$. Let $t_{0} \in \mathbb{R}$ be any time, and consider the class $\mathcal{U}(V, t_{0})$ of all families $U = \{ U(t,s) \mid s,t \in \mathbb{R}, s \leqslant t \leqslant t_{0} \}$ of mappings $U(t,s) \colon V \to V$ such that
\begin{enumerate}
    \item $U(s,s) = Id$ for all $s \leqslant t_{0}$;
    \item $U(t,r) U(r,s) = U(t,s)$ for any $s \leqslant r \leqslant t \leqslant t_{0}$.
\end{enumerate}
In this setting, a natural way to introduce a smoothing property is to consider only those families $U \in \mathcal{U}(V, t_{0})$ such that there exists a positive time span $\tau_{0}$ for which
\begin{equation}\label{E:smoothing_class}
    U(t, t-\tau_{0}) \in \mathbb{S}_{\delta, K}(B)
\end{equation}
holds for all $t \leqslant t_{0}$. Thanks to Theorem~\ref{T:discrete_time_attractor}, for any $t \leqslant t_{0}$ the family $U^{t} \in \mathcal{U}_{d}(V,0)$ given by
\begin{equation*}
    U^{t}(m,n) \eqdef U(t + m \tau_{0}, t + n\tau_{0}) \quad \text{for all $n \leqslant m \leqslant 0$}
\end{equation*}
possesses a discrete time exponential pullback attractor.

In order to obtain a satisfactory dynamical description of the system also in the continuous-time case, we will need some additional assumptions on the time regularity and continuous dependence of the family $U \in \mathcal{U}(V, t_{0})$. In particular we will assume
\begin{hypothesis_bis}
    \item \label{H:continuity_forcing} Continuity w.r.t.\ the forcing terms: there exist real positive constants $\Const_{0}$, $\epsilon_{0}$ and $\gamma$ such that $\epsilon_{0} \leqslant \tau_{0}$ and that for all $t \leqslant t_{0}$, $\tau_{0} \leqslant r \leqslant 2 \tau_{0}$, $0 \leqslant s \leqslant \epsilon_{0}$ and $v \in \mathcal{O}_{\delta}(B)$
        \begin{equation*}
            \norm{}{U(t, t-r)v - U(t-s, t-r-s)v} \leqslant \Const_{0} |s|^{\gamma}.
        \end{equation*}
    \item \label{H:back_continuous_dependence} Past continuous dependence on initial data: there exists a positive constant $\Const_{B}$ such that
        \begin{equation*}
            \norm{}{U(t, t-s)v - U(t, t-s)w} \leqslant \Const_{B} \norm{}{v-w}
        \end{equation*}
        for all $v, w \in B$ and any $t \leqslant t_{0}$, $0 \leqslant s \leqslant 2 \tau_{0}$.
    \item \label{H:time_continuity} Time continuity of solutions: there exist positive constants $\Const'_{0}$ and $\gamma'$ such that for all $t \leqslant t_{0}$, $\tau_{0} \leqslant r \leqslant 2 \tau_{0}$, $0 \leqslant s \leqslant \epsilon_{0}$ and $v \in B$
        \begin{equation*}
            \norm{}{U(t, t-r)v - U(t-s, t-r)v} \leqslant \Const'_{0} |s|^{\gamma^{'}}.
        \end{equation*}
\end{hypothesis_bis}

We can now state the main result on exponential pullback attractors
\begin{theorem}[{\cite[Theorem~2.2]{Langa2010}}]\label{T:continuous_time_attractor}
    If $U \in \mathcal{U}(V, t_{0})$ satisfies~\eqref{E:smoothing_class} and Assumption~\eqref{H:back_continuous_dependence}, with $B \subset V$ bounded and closed in $V$, then the family $\mathcal{M}_{U} = \{ \mathcal{M}_{U}(t) \mid t \leqslant t_{0} \}$, defined by
    \begin{equation*}
        \mathcal{M}_{U}(t) \eqdef \bigcup_{s \in [0, \tau_{0}]} U(t, t-s-\tau_{0}) \mathcal{M}_{U^{t-s-\tau_{0}}}(0) \qquad \text{for all $t \leqslant t_{0}$},
    \end{equation*}
    satisfies
    \begin{enumerate}[a)]
        \item $U(t, \tau) \mathcal{M}_{U}(\tau) \subset \mathcal{M}_{U}(t)$ for all $\tau \leqslant t \leqslant t_{0}$,
        \item $\mathcal{M}_{T_{-\tau}U}(t) = \mathcal{M}_{U}(t-\tau)$ for all $\tau \geqslant 0$ and any $t \leqslant t_{0}$, where $T_{-\tau} U(t,s) \eqdef U(t-\tau, s-\tau)$,
        \item for all $\tau \geqslant 0$ and any $t \leqslant t_{0}$
            \begin{equation*}
                \dst_{V}(U(t, t-\tau)B, \mathcal{M}_{U}(t)) \leqslant \Const e^{-\tilde{\alpha}\tau},
            \end{equation*}
        \item if, for any $D \subset V$ bounded, there exists a time $s_{D} \geqslant 0$ such that
            \begin{equation*}
                U(t, t-s_{D})D \subset B \qquad \text{for all $t \leqslant t_{0}$}
            \end{equation*}
            then
            \begin{equation}\label{E:exponential_decay}
                \dst_{V}(U(t, t-\tau)D, \mathcal{M}_{U}(t)) \leqslant \Const e^{\tilde{\alpha} s_{D}} e^{- \tilde{\alpha} \tau} \qquad \text{for all $\tau \geqslant s_{D}$ and any $t \leqslant t_{0}$.}
            \end{equation}
        \setcounter{enumi_tmp}{\value{enumi}}
    \end{enumerate}
    If, moreover, $U$ also satisfies Assumptions~\eqref{H:continuity_forcing} and~\eqref{H:time_continuity}, then
    \begin{enumerate}[a)]
        \setcounter{enumi}{\value{enumi_tmp}}
        \item $\mathcal{M}_{U}(t)$ is a compact subset of $V$, with finite fractal dimension, for all $t \leqslant t_{0}$,
        \item for all $0 \leqslant r \leqslant \epsilon_{0}$ and any $t \leqslant t_{0}$,
            \begin{equation*}
                \dst^{\sym}_{V}(\mathcal{M}_{U}(t), \mathcal{M}_{U}(t-r)) \leqslant \Const |r|^{\tilde{\gamma}}.
            \end{equation*}
    \end{enumerate}
\end{theorem}

Finally consider the case of processes on $V$. Let $U$ be a family $U = \{ U(t,s) \mid s, t \in \mathbb{R}, s \leqslant t \}$ of mappings $U(t,s) \colon V \to V$ such that
\begin{enumerate}
    \item $U(s,s) = Id$ for all $s \in \mathbb{R}$;
    \item $U(t,r) U(r,s) = U(t,s)$ for any $s \leqslant r \leqslant t$.
\end{enumerate}
This corresponds to a dynamical system defined not only up to the present time~$t_{0}$, but also for positive times. Considering processes corresponds to investigating what the eventual fate of the system under scrutiny will be. Therefore, it is interesting to investigate the relation between this eventual fate and the present state of the system, which is itself, in a way of speaking, the outcome of an arbitrary long evolution.

We will need the following additional assumption, which is a slight modification of~\eqref{H:back_continuous_dependence}
\begin{hypothesis_bis}
    \item \label{H:forward_continuous_dependence} Future continuous dependence on initial data: for any $t > t_{0}$ and $D_{1}$, $D_{2}$ bounded subsets of $V$, there exists a positive constant $L(t, D_{1}, D_{2})$ such that
        \begin{equation*}
            \norm{}{U(t, t_{0})v - U(t, t_{0})w} \leqslant L(t, D_{1}, D_{2}) \norm{}{v-w} \qquad \text{for all $v \in D_{1}$, $w \in D_{2}$}.
        \end{equation*}
\end{hypothesis_bis}

\begin{theorem}[{\cite[Theorem~2.3]{Langa2010}}]\label{T:process_attractor}
    Assume that $U$ is a process on $V$ and, for some $t_{0} \in \mathbb{R}$, the subfamily of $U$ given by the operators $U(t,s)$ when $s \leqslant t \leqslant t_{0}$ satisfies~\eqref{E:smoothing_class} and Assumption~\eqref{H:back_continuous_dependence}, with $B \subset V$ bounded and closed in $V$. Under these assumptions, the family $\widetilde{\mathcal{M}}_{U} = \{ \widetilde{\mathcal{M}}_{U}(t) \mid t \in \mathbb{R} \}$ defined by
    \begin{equation*}
        \widetilde{\mathcal{M}}_{U}(t) =
        \begin{cases}
            \mathcal{M}_{U}(t)                  &\text{if $t \leqslant t_{0}$,}\\
            U(t,t_{0})\mathcal{M}_{U}(t_{0})    &\text{if $t > t_{0}$,}
        \end{cases}
    \end{equation*}
    where $\mathcal{M}_{U}$ is the family given in Theorem~\ref{T:continuous_time_attractor}, satisfies:
    \begin{enumerate}[a)]
        \item $U(t, \tau) \widetilde{\mathcal{M}}_{U}(\tau) \subset \widetilde{\mathcal{M}}_{U}(t)$, for all $\tau \leqslant t$,
        \item $\widetilde{\mathcal{M}}_{T_{-\tau}U}(t) = \widetilde{\mathcal{M}}_{U}(t-\tau)$ for all $\tau \geqslant 0$ and any $t \leqslant t_{0}$ and
            \begin{equation*}
                \widetilde{\mathcal{M}}_{T_{-\tau}U}(t) \subset \widetilde{\mathcal{M}}_{U}(t-\tau) \qquad \text{for all $\tau \geqslant 0$ and any $t > t_{0}$},
            \end{equation*}
            where $T_{-\tau}U(t,s) \eqdef U(t-\tau, s-\tau)$.
        \setcounter{enumi_tmp}{\value{enumi}}
    \end{enumerate}
    If in addition~\eqref{H:forward_continuous_dependence} holds, then
    \begin{enumerate}[a)]
        \setcounter{enumi}{\value{enumi_tmp}}
        \item if, for any $D \subset V$ bounded, there exists a positive time $s_{D}$ such that
            \begin{equation*}
                U(t, t-s)D \subset B \qquad \text{for all $s \geqslant s_{D}$ and any $t \leqslant t_{0}$},
            \end{equation*}
            then $\widetilde{\mathcal{M}}_{U}$ satisfies~\eqref{E:exponential_decay} for all $t \leqslant t_{0}$ and
            \begin{equation*}
                \dst_{V}(U(t, t-\tau)D, \widetilde{\mathcal{M}}_{U}(t)) \leqslant \widetilde{L}(t, B, \widetilde{\mathcal{M}}_{U}(t_{0})) e^{\tilde{\alpha}(s_{D} + t - t_{0})}e^{-\tilde{\alpha} \tau}
            \end{equation*}
            for all $t > t_{0}$ and any $\tau \geqslant s_{D} + t - t_{0}$.
        \setcounter{enumi_tmp}{\value{enumi}}
    \end{enumerate}
    Moreover, if $U$ also satisfies Assumptions~\eqref{H:continuity_forcing} and~\eqref{H:time_continuity}, then
    \begin{enumerate}[a)]
        \setcounter{enumi}{\value{enumi_tmp}}
        \item $\widetilde{\mathcal{M}}_{U}(t)$ is a compact subset of $V$ with finite fractal dimension for all $t \in \mathbb{R}$,
        \item for all $0 \leqslant r \leqslant \epsilon_{0}$ and any $t \leqslant t_{0}$
            \begin{equation*}
                \dst^{\sym}_{V}(\widetilde{\mathcal{M}}_{U}(t), \widetilde{\mathcal{M}}_{U}(t-r)) \leqslant \Const |r|^{\tilde{\gamma}}.
            \end{equation*}
    \end{enumerate}
\end{theorem}

\begin{remark}
    We recall that in~\cite{Langa2010} also explicit estimates on a fractal dimension of the pullback exponential attractor have been derived. For the sake of simplicity, we neglect them here.
\end{remark}

\section{Existence results and basic energy estimate}\label{S:Existence}

In this section we recall some basic energy estimates, which are obtained naturally when proving existence of solution to system~\eqref{E:PDE}. First of all, for the sake of simplicity, we set $\epsilon = m = 1$ and we write the definition of weak solution to system~\eqref{E:PDE}.
\begin{definition}\label{D:weak_solution}
    Let $\vect{z}_{0} \eqdef (\vect{u}_{0}, \psi_{0})\in \Lpvect_{\divfree}(\Omega) \times \Hs_{(c0)}(\Omega)$ and let $\tau \in \mathbb{R}$. Then a couple $\vect{z} = (\vect{u}, \psi)$ such that
    \begin{align*}
        \vect{u}    &\in \Bochner{\Lp}{\tau}{T}{\Hsvect_{0,\divfree}(\Omega)} \cap \Bochner{\Hs}{\tau}{T}{\Hsvect[-1]_{\divfree}(\Omega)} \\
        \psi        &\in \Bochner{\Lp}{\tau}{T}{\Hs[3](\Omega)} \cap \Bochner{\Hs}{\tau}{T}{\Hs[-1](\Omega)}
    \end{align*}
    is called a \defi{weak solution} to~\eqref{E:PDE} if
    \begin{align*}
        &\duality{\dt{\vect{u}}(t)}{\vect{v}} + \duality{(\vect{u}(t) \cdot \nabla) \vect{u}(t)}{\vect{v}} + \duality{\nu \nabla \vect{u}(t)}{\nabla \vect{v}} = \duality{\mu(t) \nabla \psi(t)}{\vect{v}}\\
        &\duality{\dt{\psi}(t)}{\phi} + \duality{\vect{u}(t) \cdot \nabla \psi(t)}{\phi} = - \Ltwoprod{\nabla \mu(t)}{\nabla \phi}
    \end{align*}
    hold for a.e.\ $t \in [\tau,T]$, for all $\vect{v} \in \mathcal{V}$ and for all $\phi$ in $\Cont[\infty](\Omega)$, if
    \begin{equation*}
        \mu(t) = f(\psi(t)) - \lapl \psi(t)
    \end{equation*}
    holds for a.e.\ $t \in [\tau,T]$ in $\Hs(\Omega)$ with $\mu \in \Bochner{\Lp}{\tau}{T}{\Hs(\Omega)}$ and if
    \begin{equation*}
        \lim_{t \to \tau^{+}} \vect{u}(t)  = \vect{u}_{0} \quad \text{in $\Lpvect_{\divfree}(\Omega)$,} \qquad
        \lim_{t \to \tau^{+}} \psi(t)      = \psi_{0}     \quad \text{in $\Hs_{(c0)}(\Omega)$}.
    \end{equation*}
\end{definition}

The well-posedness for problem~\eqref{E:PDE}-\eqref{E:BC} is justified in a suitable Galerkin scheme, thanks to the following a priori estimates and the subsequent Lemma \ref{lemma93}(see e.g.~\cite{Boyer1999,Gal2010}).

\begin{theorem}\label{T:existence}
    Let assumptions~\eqref{Hp:regularity}--\eqref{Hp:constants} hold. If $\vect g$ satisfies~\eqref{Hp:gL2loc} and $\vect{z}_{0} \eqdef (\vect{u}_{0},\psi_{0})\in \myH0$, then there exists a unique weak solution $\vect{z}(t) = (\vect{u}(t), \psi(t))$ departing at time $\tau$ from the initial datum  $\vect{z}_{0}$.
\end{theorem}

We now obtain the first basic energy estimates that will be the basis for the estimates of the following sections.

\begin{lemma}\label{lemmabase}
    If $\vect g$ satisfies \eqref{Hp:gL2loc} and $\vect{z}(t) = (\vect{u}(t), \psi(t))$ is the solution departing at time $\tau$ from the initial datum  $\vect{z}_{0} \eqdef (\vect{u}_{0},\psi_{0})\in \myH0$, denoting by $\mu(t)$ the corresponding chemical potential, there holds
    \begin{align}\label{basic}
        & \Lpnorm{\vect{u}(t)}^2 + \Lpnorm{\nabla \psi(t)}^2 +2 \Lpnorm[1]{F(\psi(t))} + \int_{\tau}^{t} \left[ \nu \norm{\Hsvect_{0, \divfree}(\Omega)}{\vect{u}(s)}^2 +\Lpnorm{\nabla\mu(s)}^{2} \right] \, \mathrm{d}s\\
        \leqslant & \Lpnorm{\vect{u}_{0}}^{2} + \Lpnorm{\nabla \psi_{0}}^{2} + 2 \Lpnorm[1]{F(\psi_0)} + \Const \int_{\tau}^{t} \Lpnorm{\vect{g}(s)}^{2} \, \mathrm{d}s. \nonumber
    \end{align}
    Besides, there hold
    \begin{align}\label{intDelta22}
        &\int_{\tau}^{t} \Lpnorm{\nabla \psi(s)}^{2} \, \mathrm{d}s + \int_{\tau}^{t} \Lpnorm[1]{F(\psi(s))} \, \mathrm{d}s + \int_{\tau}^{t} \Lpnorm{\lapl \psi(s)}^{2} \, \mathrm{d}s\\
        \leqslant & \Const \left( \Lpnorm{\vect{u}_{0}}^{2} + \Lpnorm{\nabla \psi_{0}}^{2} + 2 \Lpnorm[1]{F(\psi_{0})} + \int_{\tau}^{t} \Lpnorm{\vect{g}(s)}^{2} \, \mathrm{d}s \right) + \Const (t - \tau) \nonumber
    \end{align}
    as well as
    \begin{align}\label{intDelta24}
        & \int_{\tau}^{t} \Lpnorm{\lapl \psi(s)}^{4} \, \mathrm{d}s \\
        \leqslant & \Const \left( \Lpnorm{\vect{u}_{0}}^{2} + \Lpnorm{\nabla \psi_{0}}^{2} + 2 \Lpnorm[1]{F(\psi_{0})} + \int_{\tau}^{t} \Lpnorm{\vect{g}(s)}^{2} \, \mathrm{d}s \right)^{2} \nonumber \\
        & {}+ \Const (t-\tau) \left( \Lpnorm{\vect{u}_{0}}^{2} + \Lpnorm{\nabla \psi_{0}}^{2} + 2 \Lpnorm[1]{F(\psi_{0})} + \int_{\tau}^{t} \Lpnorm{\vect{g}(s)}^{2} \, \mathrm{d}s \right) .\nonumber
    \end{align}
\end{lemma}

\begin{proof}
    In order to obtain our first (dissipative) a priori estimate, we multiply the first equation in~\eqref{E:PDE} by $\vect{u}$ and the third by $\mu$. Recalling the antisymmetric property of the convective term in the Navier Stokes equation and exploiting the useful vector identity
    \begin{align*}
        \Ltwoprod{ \dt{\psi}}{\mu} & = - \Ltwoprod{\dt{\psi}}{\lapl \psi} + \Ltwoprod{f(\psi)}{ \dt{\psi}}\\
        = &\frac{1}{2} \timeder \Lpnorm{\nabla \psi}^{2} + \Ltwoprod{f(\psi)}{ \dt{\psi}}\\
        = &\frac{1}{2} \timeder \left( \Lpnorm{\nabla \psi}^{2} + 2 \Lpnorm[1]{ F(\psi)} \right),
    \end{align*}
    we obtain
    \begin{equation}\label{E:stima_energia_step1}
        \frac{1}{2} \timeder \left( \Lpnorm{\vect{u}}^{2} + \Lpnorm{\nabla \psi}^{2} + 2 \Lpnorm[1]{ F(\psi)} \right) +  \nu \Lpnorm{\nabla \vect{u}}^{2} + \Lpnorm{\nabla \mu}^{2} = \Ltwoprod{\vect{g} }{\vect{u}}.
    \end{equation}
    Recalling Poincar\'e inequality for $\vect{u}$, integrating this formula with respect to time, we deduce~\eqref{basic}.

    We now have to ``complete the norms'' on the left hand side of~\eqref{E:stima_energia_step1}. From the definition of the chemical potential $\mu$ (i.e.\ from the fourth equation in~\eqref{E:PDE}) we have
    \begin{equation*}
        \Ltwoprod{\mu}{\psi} = \Lpnorm{\nabla \psi}^{2} + \Ltwoprod{ \psi}{f(\psi)}.
    \end{equation*}
    Since, by assumption, $\psi$ is mean free, we also deduce
    \begin{equation*}
        \Ltwoprod{\mu}{\psi} = \Ltwoprod{\mu - \mean{\mu}}{\psi} \leqslant \frac12\Lpnorm{\nabla \mu}^{2} + \Const \Lpnorm{\psi}^{2},
    \end{equation*}
    where $\Const$ is a constant, which only depends on the domain $\Omega$. From assumption~\eqref{Hp:splitting} on the potential $F$ we further deduce
    \begin{equation*}
        \Ltwoprod{f(\psi)}{\psi} = \Ltwoprod{f_{0}(\psi)}{\psi} - 2 \alpha \Lpnorm{\psi}^{2},
    \end{equation*}
    being $f_0=F'_0$.
    Taking into account the convexity of $F_{0}$ we can also bound the right hand side of this identity from below:
    \begin{equation*}
        \Ltwoprod{f_{0}(\psi)}{\psi} \geqslant \Lpnorm[1]{ F_{0}(\psi) - F_{0}(0)}.
    \end{equation*}
    Putting the last four estimates together and recalling Assumption~\eqref{Hp:constants}, we obtain
    \begin{align*}
        &\Lpnorm{\nabla \mu}^{2} + \Const \Lpnorm{\psi}^{2}\\
        \geqslant& \Lpnorm{\nabla \psi}^{2} + \Lpnorm[1]{ F_{0}(\psi) - F_{0}(0)} - 2 \alpha \Lpnorm{\psi}^{2}\\
        =& \Lpnorm{\nabla \psi}^{2} + \Lpnorm[1]{ F(\psi)} - \alpha \Lpnorm{\psi}^{2}.
    \end{align*}
    Therefore we get
    \begin{align*}
        \Lpnorm{\nabla \psi}^{2} +\Lpnorm[1]{ F(\psi)}
        \leqslant \Lpnorm{\nabla \mu}^{2} + \Const \Lpnorm{\psi}^{2}
        \leqslant \Lpnorm{\nabla \mu}^{2} + \delta \Lpnorm[2+q]{\psi}^{2+q} + \Const,
    \end{align*}
    where $q$ is a positive real number, $\delta$ is a (small) positive constant, which will be determined later, and $\Const$ is a positive constant, which depends only on the domain $\Omega$ and is independent of the exponent $q$ as soon as $q \geqslant \overline{q} > 0$.
    \begin{remark}
        We observe that under assumption~\eqref{Hp:polynomial_growth} we immediately have $q \geqslant 2$ so that in our case the constant $\Const$ really depends only on $\Omega$.
    \end{remark}

    By adding this last estimate and~\eqref{E:stima_energia_step1} together, choosing $\delta$ small enough, we finally deduce the basic energy estimate for system~\eqref{E:PDE}-\eqref{E:BC}
    \begin{multline}\label{E:basic_energy_estimate}
        \timeder \left( \Lpnorm{\vect{u}}^{2} + \Lpnorm{\nabla \psi}^{2} + 2 \Lpnorm[1]{ F(\psi)} \right) +
        \Const \left( \Lpnorm{\nabla \vect{u}}^{2} + \Lpnorm{\nabla \psi}^{2} +2\Lpnorm[1]{ F(\psi)} + \Lpnorm{\nabla \mu}^{2} \right)\\
        \leqslant \Const \left( 1 + \Lpnorm{\vect{g}}^{2} \right).
    \end{multline}
    Integrating with respect to time from $\tau$ to $t$, we then obtain the first part of estimate~\eqref{intDelta22}. Noticing that
    \begin{equation*}
        \duality{\nabla \mu}{\nabla\psi} = - \duality{\mu}{\lapl \psi} = \Lpnorm{\lapl \psi}^{2} - \duality{f(\psi)}{\lapl \psi} = \Lpnorm{\lapl \psi}^{2} + \duality{f'(\psi) \nabla \psi}{\nabla \psi} \geqslant  \Lpnorm{\lapl \psi}^{2} - 2 \alpha \Lpnorm{\nabla \psi}^{2},
    \end{equation*}
    we have
    \begin{equation}\label{deltainq}
        \Lpnorm{\lapl \psi}^{2} \leqslant \Lpnorm{\nabla \mu} \Lpnorm{\nabla \psi} + 2\alpha \Lpnorm{\nabla \psi}^2,
    \end{equation}
    which, integrated in time, on account of the above estimate \eqref{basic}, gives the second part of estimate~\eqref{intDelta22}.

    In order to prove~\eqref{intDelta24}, we square \eqref{deltainq}, obtaining
    \begin{equation*}
        \Lpnorm{\Delta \psi}^{4} \leqslant \Const ( \Lpnorm{\nabla \mu}^{2} \Lpnorm{\nabla \psi}^{2} + \Lpnorm{\nabla \psi}^{4} ) \leqslant \Const (\Lpnorm{\nabla \mu}^{2} \Lpnorm{\nabla \psi}^{2} + \Lpnorm{\nabla \psi}^{2} \Lpnorm{\lapl \psi}^{2}).
    \end{equation*}
    By an integration in time, in view of~\eqref{basic} and~\eqref{intDelta22} we accomplish our purpose.
\end{proof}

\begin{corollary}
\label{cor:diss}
    If $\vect g$ satisfies~\eqref{Hp:gL2loc} and~\eqref{Hp:gM2} and $\vect{z}(t) = (\vect{u}(t), \psi(t))$ is the solution departing at time $\tau$ from the initial datum  $\vect{z}_{0} \eqdef (\vect{u}_{0},\psi_{0})\in \myH0$, the following dissipative estimate holds
    \begin{align}\label{E:dissipative0}
        & \Lpnorm{\vect{u}(t)}^{2} + \Lpnorm{\nabla \psi(t)}^{2} + 2 \Lpnorm[1]{F(\psi(t))}\\
        \leqslant & \left( \Lpnorm{\vect{u}_{0}}^{2} + \Lpnorm{\nabla \psi_{0}}^{2} + 2 \Lpnorm[1]{F(\psi_{0})} \right) e^{-\Const (t - \tau)} + \Const \left( 1 + M_{\vect g}(t)\right),\qquad \forall t\geqslant \tau.\nonumber
    \end{align}
\end{corollary}

\begin{proof}
    The dissipative estimate easily follows from the basic energy estimate~\eqref{E:basic_energy_estimate} using Poincar\'e's and  Gronwall's inequalities as well as the known estimate
    \begin{align*}
        &e^{-\Const t} \int_{\tau}^{t} e^{\Const s} \Lpnorm{\vect{g}(s)}^{2} \, \mathrm{d}s\\
        \leqslant & e^{-\Const t} \sum_{n=0}^{\infty} \int_{t-(n+1)}^{t-n} e^{\Const s}\Lpnorm{\vect{g}(s)}^{2} \, \mathrm{d}s \leqslant e^{-\Const t} \sum_{n=0}^{\infty} e^{\Const (t-n)} \sup_{r \leqslant t} \int_{r-1}^{r} \Lpnorm{\vect{g}(s)}^{2} \, \mathrm{d}s\\
        \leqslant & \Const \sup_{r \leqslant t} \int_{r-1}^{r} \Lpnorm{\vect{g}(s)}^{2} \, \mathrm{d}s,
    \end{align*}
    which holds for $\vect{g} \in \Bochner{\Lp_{\uloc}}{-\infty}{t}{\Lpvect(\Omega)}$.
\end{proof}

\begin{remark}
    A bound on $\nabla \lapl \psi$ in $\Bochner{\Lp}{\tau}{T}{\Lpvect(\Omega)}$ can also be easily deduced by computing the $\Lp$ norm of the gradient of the equation for the chemical potential $\mu$ in~\eqref{E:PDE} thus leading to the regularity of the order parameter field required by Definition~\ref{D:weak_solution}. However, this estimate cannot be easily made uniform with respect to the shape of the potential $F$ (and in particular with respect to the growth exponent~$p$).
\end{remark}

\begin{remark}
    From the above computations we deduce the following regularity for weak solutions of system~\eqref{E:PDE}
    \begin{align*}
        \vect{u}    &\in \Bochner{\Lp[\infty]}{\tau}{T}{\Lpvect_{\divfree}(\Omega)} \cap \Bochner{\Lp}{\tau}{T}{\Hsvect_{0, \divfree}(\Omega)}\\
        \psi        &\in \Bochner{\Lp[\infty]}{\tau}{T}{\Hs(\Omega)} \cap \Bochner{\Lp[4]}{\tau}{T}{\Hs[2](\Omega)}\\
        F(\psi)     &\in \Bochner{\Lp[\infty]}{\tau}{T}{\Lp[1](\Omega)}\\
        \nabla \mu  &\in \Bochner{\Lp}{\tau}{T}{\Lpvect (\Omega)}
    \end{align*}
    for any $T \in \mathbb{R}$, $T > \tau$.
\end{remark}

\section{Higher regularity estimates}\label{S:higher_order}

In order to obtain estimates having uniform structure with respect to the growth exponent of $f$, we henceforth assume that $F$ satisfies~\eqref{Hp:polynomial_growth}. Although all exponents and norms that appear in this and in the following sections are independent of $p$, the general constant $\Const$ will quickly become larger as $p$ grows.

In particular, Assumption~\eqref{Hp:polynomial_growth} and Lemma~\ref{lemmabase} imply
\begin{align}\label{p+3}
    \Lpnorm[\frac{p+3}{p+2-k}]{f^{(k)}(\psi(t))} & \leqslant \Const( \Lpnorm[1]{F(\psi(t))} + 1)\\
    & \leqslant \Const \left( \Lpnorm{\vect{u}_{0}}^{2} + \Lpnorm{\nabla \psi_{0}}^2 + 2 \Lpnorm[1]{F(\psi_{0})} + \int_{\tau}^{t} \Lpnorm{\vect{g}(s)}^{2} \, \mathrm{d}s+ 1 \right),\nonumber
\end{align}
being $(\vect{u}(t),\psi(t))$ the solution to~\eqref{E:PDE}-\eqref{E:BC} departing from $(\vect{u}_0,\psi_0) \in \mathcal{H}_{0}$ at time $\tau$.

The goal of this section is to improve ``by one order'' the basic regularity result already obtained. In particular, under suitable assumptions, we will get to $\vect{u} \in \bochner{\Lp[\infty]}{\Hsvect_{0, \divfree}(\Omega)} \cap \bochner{\Lp}{\Hsvect[2]_{0, \divfree}(\Omega)}$ and $\psi \in \bochner{\Lp[\infty]}{\Hs[2](\Omega)} \cap \bochner{\Lp}{\Hs[4](\Omega)}$. This will be achieved in several steps gaining before spatial regularity for $f(\psi)$ and $\mu$ and later time regularity as well: first in Lemma \ref{lemma0} we will deduce $f(\psi) \in \bochner{\Lp}{\Lp(\Omega)}$ and $\mu \in \bochner{\Lp}{\Lp(\Omega)}$; then $f(\psi) \in \bochner{\Lp}{\Lp[q](\Omega)}$ and $\lapl\psi \in \bochner{\Lp}{\Lp[q](\Omega)}$ for any $q \geqslant 1$, as shown in Lemma \ref{lemmaint0}; this will give $\mu \in \bochner{\Lp[\infty]}{\Lp(\Omega)}$ (cf. Lemma \ref{lemmafurtherintmu}) and the final result (see Lemma \ref{lemma:ordine superiore}).

\begin{notation}
    In order to simplify notation, we will denote by $\dataConst$ the quantity
    \begin{equation*}
        \dataConst \eqdef 1 + \Lpnorm{\vect{u}_{0}}^{2} + \Lpnorm{\nabla \psi_{0}}^{2} + 2 \Lpnorm[1]{F(\psi_{0})} + \int_{\tau}^{t} \Lpnorm{\vect{g}(s)}^{2} \, \mathrm{d}s,
    \end{equation*}
    which depends only on the initial data $\vect{u}_{0}$, $\psi_{0}$, on the forcing term $\vect{g}$ and on the times $t$ and $\tau$. Besides, $C$ stands for a generic positive constant depending only on $\Omega$ and possibly on $p$ and is allowed to vary even in the same line.
\end{notation}

\begin{lemma}\label{lemma0}
    If $\vect{z}(t) = (\vect{u}(t), \psi(t))$ is the solution departing at time $\tau$ from the initial datum  $\vect{z}_{0} \doteq (\vect{u}_{0}, \psi_{0}) \in \myH0$, denoting by $\mu(t)$ the corresponding chemical potential, there holds
    \begin{equation}\label{E:regularity_1}
        \int_{\tau}^{t} (\Lpnorm{f(\psi(s)}^{2} + \Lpnorm{\mu(s)}^{2}) \, \mathrm{d}s \leqslant \Const \dataConst^{2} + \Const (t-\tau) \dataConst
    \end{equation}
    for any $t\geqslant \tau$, $\tau\in \mathbb{R}$.
\end{lemma}

\begin{proof}
    Remark~2.1 allows to bound the mean value of $f(\psi)$ as
    \begin{equation*}
        |\l f(\psi)\r|\leqslant \Const |f(\psi)|_{\frac{p+3}{p+2}} \leqslant \Const (1 + \Lpnorm[1]{F(\psi)}),
    \end{equation*}
    for some $\Const>0$ depending on $p$ only through the constant $c_0$ in Assumption~\eqref{Hp:polynomial_growth}. Recalling the equation defining the chemical potential in~\eqref{E:PDE} and estimate~\eqref{deltainq}, we further deduce
    \begin{equation*}
        \Lpnorm{f(\psi)- \mean{f(\psi)}}^{2} \leqslant 2 \Lpnorm{\mu - \mean{\mu}}^{2} + 2 \Lpnorm{\lapl \psi}^{2} \leqslant \Const \Lpnorm{\nabla \mu}^{2} + 4\alpha \Lpnorm{\nabla \psi}^{2}.
    \end{equation*}
    Therefore, we gain full control on the $\Lp$-norm of $f(\psi)$, bounding its time integral as
    \begin{align*}
        & \int_{\tau}^{t} \Lpnorm{f(\psi(s))}^{2} \, \mathrm{d}s\\
        \leqslant & \Const \int_{\tau}^{t} \Lpnorm{\nabla \mu(s)}^{2} \, \mathrm{d}s + 4\alpha \int_{\tau}^{t} \Lpnorm{\nabla \psi(s)}^{2}\, \mathrm{d}s + \Const \int_{\tau}^{t}  (1 + \Lpnorm[1]{F(\psi(s))})^2 \, \mathrm{d}s\\
        \leqslant & \Const \dataConst^{2} + \Const (t - \tau) \dataConst,
    \end{align*}
    where we used~\eqref{basic} and~\eqref{intDelta22} from Lemma~\ref{lemmabase}. The second part of estimate~\eqref{E:regularity_1} follows from
    \begin{equation*}
        \Lpnorm{\mu}^{2} \leqslant 2 (\Lpnorm{\lapl \psi}^{2} + \Lpnorm{f(\psi)}^{2}),
    \end{equation*}
    the above bound and~\eqref{intDelta22}.
\end{proof}
As announced before, the integrability of $f(\psi)$ can be further improved in two steps.

\begin{lemma}\label{lemmaint0}
    If $\vect{z}(t) = (\vect{u}(t),\psi(t))$ is the solution departing at time $\tau$ from the initial datum  $\vect{z}_{0} \doteq (\vect{u}_{0}, \psi_{0}) \in \myH0$, then, for any $b > 0$ there exists $C_{b} > 0$ such that
    \begin{equation*}
        \int_{\tau}^{t} \left( \Lpnorm[b+2]{f(\psi(s))}^{2} + \Lpnorm[b+2]{\lapl \psi(s)}^{2} \right) \, \mathrm{d}s \leqslant \Const_{b} \dataConst^{2} + \Const_{b} (t - \tau) \dataConst,
    \end{equation*}
    meaning that $f(\psi)$, $\lapl \psi \in \Bochner{\Lp}{\tau}{t}{\Lp[b+2](\Omega)}$, for  any $b > 0$.
\end{lemma}

\begin{remark}
    We note that this estimate extends to singular functional $f$, without appealing to approximation arguments as in \cite{Abels2009a} but with the same order of control.
\end{remark}

\begin{proof}
    Multiplying the equation for the chemical potential by $f(\psi)|f(\psi)|^{b}$ and integrating over $\Omega$, we have
    \begin{equation}\label{E:regularity_2_multiplication_by_f_b+1}
        \duality{\mu}{f(\psi) |f(\psi)|^{b}} = \Lpnorm[b+2]{f(\psi)}^{b+2} - \duality{\lapl \psi}{f(\psi) |f(\psi)|^{b}}.
    \end{equation}
    We now exploit assumption~\eqref{Hp:splitting} on~$F$, proving after an integration by parts
    \begin{equation*}
        -\duality{\lapl \psi}{f(\psi) |f(\psi)|^{b}} = (b+1) \duality{f'(\psi)}{|f(\psi)|^{b} |\nabla \psi|^{2}}\geqslant - 2\alpha (b+1) \duality{|f(\psi)|^{b}}{|\nabla \psi|^{2}}.
    \end{equation*}
    Replacement of this estimate in~\eqref{E:regularity_2_multiplication_by_f_b+1} above leads to
    \begin{equation*}
        \Lpnorm[b+2]{f(\psi)}^{b+2} \leqslant 2\alpha (b+1) \duality{|f(\psi)|^{b}}{|\nabla \psi|^{2}} + \duality{|\mu|}{|f(\psi)|^{b+1}}.
    \end{equation*}
    H\"older's and Young's inequalities then provide
    \begin{equation*}
        \Lpnorm[b+2]{f(\psi)}^{b+2} \leqslant 2\alpha (b+1)\Lpnorm[b+2]{f(\psi)}^{b} \Lpnorm[b+2]{\nabla \psi}^{2} + \Lpnorm[b+2]{\mu} \Lpnorm[b+2]{f(\psi)}^{b+1} \leqslant \frac{1}{2} \Lpnorm[b+2]{f(\psi)}^{b+2} + \Const_{b} \left( \Lpnorm[b+2]{\nabla \psi}^{b+2} + \Lpnorm[b+2]{\mu}^{b+2} \right).
    \end{equation*}
    Recalling that from standard interpolation the inequality
    \begin{equation}\label{Lint}
        \Lpnorm[b+2]{h}^{b+2} \leqslant \Const_{b} \Lpnorm{h}^{2} \norm{\Hs}{h}^{b},
    \end{equation}
    holds, we end up with
    \begin{equation}\label{fb}
        \Lpnorm[b+2]{f(\psi)}^{b+2} \leqslant \Const_{b} \left( \Lpnorm{\nabla \psi}^{2} \Lpnorm{\lapl \psi}^{b} + \Lpnorm{\mu}^{2} \norm{\Hs}{\mu}^{b} \right).
    \end{equation}
    A further application of Young's inequality then gives
    \begin{align*}
        \Lpnorm[b+2]{f(\psi)}^{2} & \leqslant \Const_{b} \left( \Lpnorm{\nabla \psi}^{\sfrac{4}{(b+2)}} \Lpnorm{\lapl \psi}^{\sfrac{2b}{(b+2)}} + \Lpnorm{\mu}^{\sfrac{4}{(b+2)}} \Lpnorm{\nabla \mu}^{\sfrac{2b}{(b+2)}} + \Lpnorm{\mu}^{2} \right)\\
        & \leqslant \Const_{b} \left( \Lpnorm{\nabla \psi}^2+\Lpnorm{\lapl \psi}^2 + \Lpnorm{\mu}^2+\Lpnorm{\nabla \mu}^{2} \right).
    \end{align*}
    Finally, integration with respect to time, leads to
    \begin{equation*}
        \int_{\tau}^{t} \Lpnorm[b+2]{f(\psi(s))}^{2} \, \mathrm{d}s \leqslant \Const_{b} \int_{\tau}^{t} \left( \Lpnorm{\nabla \psi(s)}^{2} + \Lpnorm{\lapl \psi(s)}^{2} + \Lpnorm{\mu(s)}^{2} + \Lpnorm{\nabla \mu(s)}^{2} \right) \, \mathrm{d}s,
    \end{equation*}
    thus Lemmata~\ref{lemmabase} and~\ref{lemma0} provide the first part of the desired estimate. To complete our argument, it is sufficient to exploit the equation for the chemical potential~$\mu$ and this last estimate together with Lemmata~\ref{lemmabase} and~\ref{lemma0}:
    \begin{equation*}
        \Lpnorm[b+2]{\lapl \psi}^{2} \leqslant 2 \left( \Lpnorm[b+2]{\mu}^{2} + \Lpnorm[b+2]{f(\psi)}^{2} \right) \leqslant 2 \left( \Lpnorm{\mu}^{2} + \Lpnorm{\nabla \mu}^{2} + \Lpnorm[b+2]{f(\psi)}^{2} \right). \qedhere
    \end{equation*}
\end{proof}

\begin{remark}\label{remDeltaL4L2}
    Provided that $F$ satisfies~\eqref{Hp:polynomial_growth}, estimates~\eqref{basic} and~\eqref{Fcontrol} entail $f(\psi) \in \Bochner{\Lp[\infty]}{\tau}{t}{\Lp[(p+3)/(p+2)](\Omega)}$  (cf. Remark \ref{rem:fp+3}). Besides, the above Lemma~\ref{lemmaint0} implies $f(\psi)\in \Bochner{\Lp}{\tau}{t}{\Lp[b+2](\Omega)}$ for any $b>0$. Being $(p+3)/(p+2)>1$, by the interpolation inequality
    \begin{equation}\label{interpol}
        \Lpnorm{h} \leqslant \Lpnorm[\frac{p+3}{p+2}]{h}^{\theta} \Lpnorm[b+2]{h}^{1-\theta}, \quad \text{where $\theta = \frac{b(p+3)}{2(bp+2b+p+1)}$,}
    \end{equation}
    we deduce that, when $\displaystyle 4(1-\theta) = 4 \frac{(b+2)(p+1)}{2(bp+2b+p+1)}= 2$, that is, $b = 1+p$,
    \begin{align*}
        \int_{\tau}^{t} \Lpnorm{f(\psi(s))}^{4} \, \mathrm{d}s \leqslant & \norm{\Bochner{\Lp[\infty]}{\tau}{t}{\Lp[\frac{p+3}{p+2}](\Omega)}}{f(\psi)}^{2} \int_{\tau}^{t} \Lpnorm[p+3]{f(\psi(s)}^{2} \, \mathrm{d}s\\
        \leqslant & \Const \dataConst^{4} + \Const (t - \tau) \dataConst^{3}.
    \end{align*}
    On account of~\eqref{Fcontrol}, Lemmata~\ref{lemmabase} and~\ref{lemmaint0}, it thus follows $f(\psi)\in L^4(\tau,t;L^{2}(\Omega))$ and, in particular,
    \begin{equation}\label{intL^{2}L^4}
        \norm{\Bochner{\Lp[4]}{\tau}{t}{\Lp(\Omega)}}{f(\psi)}^4 \leqslant \Const \dataConst^{4} + \Const (t - \tau) \dataConst^{3}.
    \end{equation}
\end{remark}

\begin{lemma}\label{lemmafurtherintmu}
    If $\vect{z}(t)=(\vect{u}(t), \psi(t))$ is the solution departing at time $\tau$ from any initial datum $\vect{z}_{0} \doteq (\vect{u}_{0}, \psi_{0}) \in  \myH1$ so that $\mu_{0} \eqdef f(\psi_{0}) - \lapl \psi_{0} \in \Lp(\Omega)$, then there exists $\Const > 0$ depending only on $p$ such that the chemical potential $\mu$ is bounded in $\Bochner{\Lp[\infty]}{\tau}{T}{\Lp(\Omega)} \cap \Bochner{\Lp}{\tau}{T}{\Hs[2](\Omega)}$ for all $T > \tau$ and there hold
    \begin{equation*}
        \Lpnorm{\mu(t)}^{2} \leqslant \Const \left( \Lpnorm{\mu(\tau)}^{2}  + \dataConst^{3} + (t-\tau) \dataConst \right) e^{\Const(\dataConst^{4} + (t-\tau) \dataConst^{3})}
    \end{equation*}
    and
    \begin{multline*}
        \int_{\tau}^{t} \Lpnorm{\lapl \mu(s)}^{2} \, \mathrm{d}s
        \leqslant \Const \left( \Lpnorm{\mu(\tau)}^{2} + \dataConst^{3} + (t-\tau) \dataConst \right) \left( \dataConst^{4} + (t-\tau) \dataConst^{3} \right) e^{\Const \left( \dataConst^{4} + (t-\tau) \dataConst^{3} \right)}.
    \end{multline*}
\end{lemma}

\begin{proof}
    The evolution of the chemical potential $\mu$ is governed by
    \begin{equation}\label{evolmu}
        \dt{\mu} = f'(\psi) \lapl \mu - f'(\psi) (\vect{u} \cdot \nabla)\psi - \lapl^{2} \mu + \lapl ((\vect{u} \cdot \nabla) \psi),
    \end{equation}
    as can be seen by formally differentiating with respect to time the last equation in~\eqref{E:PDE} and by taking into account the third one. The product of this equality by $\mu$ gives rise to three terms from the right hand side: in order to exploit the lower bound on $f'$, the first one can be written as
    \begin{align*}
        &\duality{f'(\psi) \lapl \mu}{\mu}\\
        = & - \duality{f'(\psi) \nabla \mu}{\nabla \mu} - \duality{f''(\psi) \nabla \psi}{\mu \nabla \mu}\\
        = & - \duality{f'(\psi) \nabla \mu}{\nabla \mu} - \frac{1}{2} \duality{f''(\psi) \nabla \psi}{\nabla (\mu^{2})}\\
        = & - \duality{f'(\psi) \nabla \mu}{\nabla \mu} + \frac{1}{2} \duality{f''(\psi) \lapl \psi}{\mu^{2}} + \frac{1}{2} \duality{f'''(\psi) \vectnorm{\nabla \psi}^{2}}{\mu^{2}}.
    \end{align*}
    Thanks to the incompressibility condition, the second one reads as
    \begin{equation*}
        -\duality{f'(\psi)(\vect{u} \cdot \nabla) \psi}{\mu} = \duality{f(\psi)}{\vect{u} \cdot \nabla \mu}.
    \end{equation*}
    Finally, noticing that the third equation in \eqref{E:PDE} and the boundary conditions \eqref{E:BC} imply $\dnu{\left( \lapl \mu-\vect u\cdot \nabla \psi\right)} = \dnu{\dt{\psi}} = 0$ on $\partial\Omega$, the last term is
    \begin{align*}
        -\duality{ \lapl^2 \mu-\lapl (\vect u\cdot \nabla\psi)}{\mu}=-\Lpnorm{\lapl \mu}^2+\duality{\vect u\cdot \nabla\psi}{\lapl \mu}.
    \end{align*}
    These computations lead us to
    \begin{multline}\label{E:mu_higher_order}
        \frac{1}{2} \timeder \Lpnorm{\mu}^{2} + \Lpnorm{\lapl \mu}^{2}\\
        = -\duality{f'(\psi) \nabla \mu}{\nabla \mu} + \frac{1}{2} \duality{f''(\psi) \lapl \psi}{\mu^{2}} + \frac{1}{2} \duality{f'''(\psi) \vectnorm{\nabla \psi}^{2}}{\mu^{2}} + \duality{f(\psi) \vect{u}}{\nabla \mu}\\
        {} + \duality{(\vect{u} \cdot \nabla) \psi}{\lapl \mu} .
    \end{multline}
    By Assumption~\eqref{Hp:splitting} on the potential $F$, the first term on the right hand side of identity~\eqref{E:mu_higher_order} is easily controlled by $2 \alpha \Lpnorm{\nabla \mu}^{2}$ while the last one can be bounded by
    \begin{align*}
        &\left| \duality{(\vect{u} \cdot \nabla) \psi}{\lapl \mu}\right|\\
        \leqslant & \Lpnorm[4]{\vect{u}} \Lpnorm[4]{\nabla \psi} \Lpnorm{\lapl \mu} \\
        \leqslant & \Const \Lpnorm{\vect{u}}^{\sfrac{1}{2}} \Lpnorm{\nabla \vect{u}}^{\sfrac{1}{2}} \Lpnorm{\nabla \psi}^{\sfrac{1}{2}}
        \Lpnorm{\lapl \psi}^{\sfrac{1}{2}} \Lpnorm{\lapl \mu}\\
        \leqslant & \frac{1}{8} \Lpnorm{\lapl \mu}^{2} + \Const \Lpnorm{\vect{u}} \Lpnorm{\nabla \vect{u}} \Lpnorm{\nabla \psi} \Lpnorm{\lapl \psi}.
    \end{align*}
    We are left to consider the other terms in~\eqref{E:mu_higher_order}. Having in mind~\eqref{p+3}, we prove
    \begin{align*}
        & \frac{1}{2} \left|\duality{ f''(\psi)\lapl\psi}{ \mu^2}\right|\\
        \leqslant & \frac12\Lpnorm[\frac{p+3}{p}]{f''(\psi)}\Lpnorm[{p+3}]{\lapl \psi}\Lpnorm[{p+3}]{\mu}^2\\
        \leqslant & \Const \Lpnorm[\frac{p+3}{p}]{f''(\psi)}\Lpnorm[p+3]{\lapl\psi} \norm{\Hs}{\mu}^2\\
        \leqslant & \Const \Lpnorm[\frac{p+3}{p}]{f''(\psi)} |\lapl\psi|_{p+3}(\Lpnorm{\mu}^2+\Lpnorm{\mu}\Lpnorm{\lapl\mu})\\
        \leqslant & \frac{1}{8}\Lpnorm{\lapl\mu}^2 + \Const\left( \Lpnorm[\frac{p+3}{p}]{f''(\psi)}^2+1\right)(\Lpnorm[p+3]{\lapl\psi}^2+1)\Lpnorm{\mu}^2.
    \end{align*}
    Analogously we have
    \begin{align*}
        & \frac{1}{2} \left|\duality{ f'''(\psi)|\nabla\psi|^2}{ \mu^2}\right|\\
        \leqslant & \frac12\Lpnorm[\frac{p+3}{p-1}]{f'''(\psi)}\Lpnorm[p+3]{\nabla \psi}^2\Lpnorm[p+3]{\mu}^2 \\
        \leqslant & \Const \Lpnorm[\frac{p+3}{p-1}]{f'''(\psi)}\Lpnorm{\lapl\psi}^2 \norm{\Hs}{\mu}^2 \\
        \leqslant & \Const \Lpnorm[\frac{p+3}{p-1}]{f'''(\psi)}\Lpnorm{\lapl\psi}^2 (\Lpnorm{\mu}^2+\Lpnorm{\mu}\Lpnorm{\lapl\mu}) \\
        \leqslant & \frac{1}{8}\Lpnorm{\lapl\mu}^2 + \Const (\Lpnorm[\frac{p+3}{p-1}]{f'''(\psi)}^2 + 1)(|\lapl\psi|_2^4 +1)\Lpnorm{\mu}^2.
    \end{align*}
    There also holds
    \begin{align*}
        & \left|\duality{ f(\psi)}{\vect{u}\cdot \nabla\mu}\right|\\
        \leqslant & \Lpnorm[\frac{p+3}{p+2}]{f(\psi)}\Lpnorm[2(p+3)]{\vect{u}}\Lpnorm[2(p+3)]{\nabla\mu}\\
        \leqslant & \Const \Lpnorm[\frac{p+3}{p+2}]{f(\psi)}\Lpnorm{\nabla \vect{u}}\norm{\Hs[2]}{ \bar\mu}\\
        \leqslant & \frac{1}{8} \Lpnorm{\lapl \mu}^2+\Const \Lpnorm{\nabla\mu}^2+
         \Const\Lpnorm[\frac{p+3}{p+2}]{f(\psi)}^2\Lpnorm{\nabla \vect{u}}^2.
    \end{align*}
    Collecting the above estimates and recalling ~\eqref{p+3}, we have
    \begin{equation}\label{E:muder}
        \timeder \Lpnorm{\mu}^2+\Lpnorm{\lapl\mu}^2 \leqslant h\Lpnorm{\mu}^2+g,
    \end{equation}
    where
    \begin{align*}
        h & =\Const (1+\Lpnorm[1]{F(\psi)})^2(1+\Lpnorm{\lapl\psi}^4 + \Lpnorm[p+3]{\lapl\psi}^2  )\\
        g& = \Const \Lpnorm{\nabla\mu}^2+\Const   \Lpnorm{\vect{u}} \Lpnorm{\nabla\vect{u}}
        \Lpnorm{\nabla\psi}
        \Lpnorm{\lapl\psi}+\Const (1+\Lpnorm[1]{F(\psi)})^2\Lpnorm{\nabla\vect{u}}^2.
    \end{align*}
    In view of~\eqref{intDelta24}, Lemmata~\ref{lemmabase} and~\ref{lemmaint0}, $h$ and $g$ are integrable quantities. Indeed, we have
    \begin{align*}
        \int_{\tau}^{t} h(s) \, \mathrm{d}s &\leqslant \Const \left( \dataConst^{4} + (t-\tau) \dataConst^{3} \right),\\
        \int_{\tau}^{t} g(s) \, \mathrm{d}s &\leqslant \Const \left( \dataConst^{3} + (t-\tau) \dataConst \right).
    \end{align*}
    By Gronwall's lemma we further deduce
    \begin{align*}
        |\mu(t)|^2_2& \leqslant \Big(|\mu(\tau)|^2_2+\int_\tau^t g(s)ds\Big) \exp{\Big(\int_\tau^t h(s)ds\Big)}
    \end{align*}
    that is,
    \begin{equation*}
        \Lpnorm{\mu(t)}^{2} \leqslant
         \Const
         \left( |\mu(\tau)|^2_2 + \dataConst^{3} + (t-\tau) \dataConst \right) e^{\Const(\dataConst^{4} + (t-\tau)\dataConst^{3})}.
    \end{equation*}
    Moreover, integrating \eqref{E:muder}, from the estimates above we also deduce
    \begin{align*}
        & \int_{\tau}^{t} \Lpnorm{\lapl \mu(s)}^{2} \, \mathrm{d}s\\
        \leqslant & \Lpnorm{\mu(\tau)}^{2} + \int_{\tau}^{t} h(s) \Lpnorm{\mu(s)}^{2} \, \mathrm{d}s + \int_{\tau}^{t} g(s) \, \mathrm{d}s\\
        \leqslant & \Const \left( |\mu(\tau)|^2_2+ \dataConst^{3} + (t-\tau) \dataConst \right) \left( \dataConst^{4} + (t-\tau) \dataConst^{3} \right) e^{\Const(\dataConst^{4} + (t-\tau) \dataConst^{3})}. \qedhere
    \end{align*}
\end{proof}

\begin{remark}\label{rem}
    The above Lemma has several consequences. First of all, from the third equation of~\eqref{E:PDE} we easily obtain $\Lpnorm{\dt{\psi}}^{2} \leqslant \Const (\Lpnorm{\lapl \mu}^{2} + \Lpnorm{\vect{u}} \Lpnorm{\nabla \vect{u}} \Lpnorm{\nabla \psi} \Lpnorm{\lapl \psi})$. Thus, Lemmata~\ref{lemmabase} and~\ref{lemmafurtherintmu} yield
    \begin{multline}\label{intptuL2}
        \int_{\tau}^{t} \Lpnorm{\dt{\psi(s)}}^{2} \, \mathrm{d}s\\
                \leqslant \Const \left( \Lpnorm{\mu(\tau)}^{2} + \dataConst^{3} + (t-\tau) \dataConst \right) \left( \dataConst^{4} + (t-\tau) \dataConst^{3} \right) e^{\Const(\dataConst^{4} (t-\tau) \dataConst^{3})}.
    \end{multline}
    Besides, by~\eqref{fb} with $b=2$, we have
    \begin{equation*}
        |f(\psi)|^4_4 \leqslant \Const (|\nabla\psi|^2_2 |\Delta\psi|^2_2+|\mu|^4_2+|\mu|^2_2|\nabla\mu|^2_2),
    \end{equation*}
    hence Lemmata~\ref{lemma0} and \ref{lemmafurtherintmu} entail $f(\psi)\in \Bochner{\Lp[4]}{\tau}{t}{\Lp[4](\Omega)}$, with
    \begin{multline}\label{effe44}
        \int_{\tau}^{t} \Lpnorm[4]{f(\psi(s))}^{4} \, \mathrm{d}s\\
        \leqslant \Const \left( \Lpnorm{\mu(\tau)}^{2} + \dataConst^{3} + (t-\tau) \dataConst \right) \left( \dataConst^{2} + (t-\tau) \dataConst \right) e^{\Const(\dataConst^{4} + (t-\tau)\dataConst^{3})}.
    \end{multline}
\end{remark}

\begin{remark}\label{remL8}
    Actually, even more uniform estimates can be deduced from the above Lemmata. For example, from~\eqref{fb}, using Ladyzhenskaja inequality and interpolation estimates, we deduce
    \begin{align*}
        \Lpnorm[4]{f(\psi)}^8 & \leqslant \Const \left(\Lpnorm{\nabla\psi}^4 \Lpnorm{\lapl\psi}^4+\Lpnorm{\mu}^8+\Lpnorm{\mu}^4 \Lpnorm{\nabla\mu}^4\right)\\
        & \leqslant \Const \left(\Lpnorm{\nabla\psi}^4 \Lpnorm{\lapl\psi}^4+\Lpnorm{\mu}^8+\Lpnorm{\mu}^6 \Lpnorm{\lapl\mu}^2\right),
    \end{align*}
    i.e.\ $f(\psi)\in \Bochner{\Lp[8]}{\tau}{t}{\Lp[4](\Omega)}$. In particular,
    \begin{align}\label{E:intfl8}
        \int_{t-1}^{t} \Lpnorm[4]{f(\psi(s))}^{8} \, \mathrm{d}s
        & \leqslant \Const \int_{t-1}^t \left( \Lpnorm{\nabla\psi(s)}^4 \Lpnorm{\lapl\psi(s)}^4+\Lpnorm{\mu(s)}^{8} + \Lpnorm{\mu(s)}^6\Lpnorm{\lapl\mu(s)}^2 \right) \, \mathrm{d}s\\
        & \leqslant \Const \left( |\mu(t-1)|^2_2+ A_{t,t-1}^{3} \right)^{4} A_{t,t-1}^{4} e^{\Const A_{t,t-1}^{4} }. \notag
    \end{align}
\end{remark}

Thanks to our assumptions on $f$ and to the previous results, we can now obtain estimates on the higher norms of the solution, which have uniform structure w.r.t.\ the shape of the potential. Dependence on the growth of the potential~$F$ is limited to the constants $\Const$, which appear in the estimate.

\begin{lemma}\label{lemma:ordine superiore}
    Given any initial datum $\vect{z}_{0} \eqdef (\vect{u}_{0}, \psi_{0}) \in \myH1$ so that $\mu_{0} \eqdef f(\psi_{0}) - \lapl \psi_{0} \in \Lp(\Omega)$, the solution departing at time $\tau$ from $\vect{z}_{0}$ satisfies
    \begin{multline}\label{E:H1_boundedness}
        \norm{\myH1}{\vect{z}(t)}^{2} \\
        \leqslant \Const \left( \norm{\myH1}{\vect{z}_0}^{2} + \Lpnorm{\mu_{0}}^{2} + \dataConst^{3} + (t-\tau) \dataConst \right) \left( \dataConst^{5} + (t-\tau) \dataConst^{4} \right) e^{\Const \left( \dataConst^{4} + (t-\tau) \dataConst^{3} \right)}
    \end{multline}
    for some constant $\Const > 0$ depending on the exponent~$p$ and on the domain~$\Omega$ but independent on the initial data. Moreover,
    \begin{multline*}
        \int_{\tau}^{t} \left(\Lpnorm{\lapl^2 \psi(s)}^{2} + \Lpnorm{\lapl \vect{u}(s)}^{2}\right) \, \mathrm{d}s\\
        \leqslant \Const \left( \norm{\myH1}{\vect{z}_0}^{2} + \Lpnorm{\mu_{0}}^{2} + \dataConst^{3} + (t-\tau) \dataConst \right) \left( \dataConst^{7} + (t-\tau) \dataConst^{6} \right) e^{\Const \left(\dataConst^{4} + (t-\tau)\dataConst^{3}\right)}.
    \end{multline*}
\end{lemma}

\begin{proof}
    In this proof, we will exploit the following inequality
    \begin{equation*}
        \vectnorm{f'(y)}+\vectnorm{f''(y)}\leqslant \Const \left(1+\vectnorm{f(y)}\right),\quad \forall y\in \mathbb R,
    \end{equation*}
    which can be easily obtained from Assumptions \eqref{Hp:coercivity}, \eqref{Hp:positivity}, \eqref{Hp:splitting}, \eqref{Hp:constants} and  \eqref{Hp:polynomial_growth} by means of the Young's inequality. In particular, we will take advantage of its straightforward consequence
    \begin{equation*}
        \Lpnorm[4]{f'(\psi)}+\Lpnorm[4]{f''(\psi)}\leqslant \Const \left(1+\Lpnorm[4]{f(\psi)}\right).
    \end{equation*}
    Adding together the product of the first equation in~\eqref{E:PDE} by $2A \vect{u} = -2 \mathbb{P} \lapl \vect{u}$ and of the third one by $2 \lapl^{2} \psi$, we obtain
    \begin{align}\label{uH1}
        & \timeder \left( \Lpnorm{\nabla \vect{u}}^{2} + \Lpnorm{\lapl \psi}^{2} \right) + 2 \nu\Lpnorm{ A \vect{u}}^2 + 2 \Lpnorm{\lapl^{2} \psi}^{2}\\
        = & 2 \duality{\vect{g}}{A \vect{u}}  - 2 \duality{(\vect{u} \cdot \nabla) \vect{u}}{A \vect{u}}+ 2 \duality{\mu \nabla \psi}{A \vect{u}} - 2 \duality{(\vect{u} \cdot \nabla) \psi}{\lapl^{2} \psi} + 2 \duality{\lapl f(\psi)}{\lapl^{2} \psi}.\nonumber
    \end{align}
    The first two terms arising from Navier-Stokes equations can be dealt with by writing
    \begin{align*}
        & \left|-2\duality{ (\vect{u}\cdot\nabla )\vect{u} }{A\vect{u}}+2\duality{ \vect{g}}{A\vect{u}}\right|\\
        \leqslant & \Const (\Lpnorm{\vect{u}}^{\sfrac12} \Lpnorm{\nabla \vect{u}} \Lpnorm{A\vect{u}}^{\sfrac32}+\Lpnorm{\vect{g}}\Lpnorm{A\vect{u}})\\
        \leqslant & \frac{\nu}{2} |A\vect{u}|^{2}_2+\Const |\vect{u}|^{2} \Lpnorm{\nabla \vect{u}}^{4} +\Const |\vect{g}|^{2}_2.
    \end{align*}
    Since $\duality{ \mu\nabla\psi}{A\vect{u}}=-\duality{ \lapl\psi\nabla\psi}{A\vect{u}}$, from the Agmon's inequality we easily have
    \begin{align*}
        & \left| 2\duality{ \mu\nabla\psi}{A\vect{u}} \right|\\
        \leqslant & 2 \Lpnorm[\infty]{\lapl\psi}\Lpnorm{\nabla\psi} \Lpnorm{A\vect{u}}\\
        \leqslant & \Const \Lpnorm{\lapl\psi}^{\sfrac12}\Lpnorm{\lapl^{2}\psi}^{\sfrac12}\Lpnorm{\nabla\psi} \Lpnorm{A\vect{u}}\\
        \leqslant & \frac{\nu}{2} \Lpnorm{A\vect{u}}^2+ \frac{1}{3} \Lpnorm{\lapl^{2}\psi}^{2}+\Const \Lpnorm{\nabla\psi}^4 \Lpnorm{\lapl\psi}^{2}.
    \end{align*}
    Then, by Ladyzhenskaja inequality and standard estimates
    \begin{align*}
        & \left| 2 \duality{(\vect{u} \cdot \nabla) \psi}{\lapl^{2} \psi} \right|\\
        \leqslant & 2 \Lpnorm[4]{\vect{u}} \Lpnorm[4]{\nabla \psi} \Lpnorm{\lapl^{2} \psi} \\
        \leqslant & \Const \Lpnorm{\vect{u}}^{\sfrac{1}{2}} \Lpnorm{\nabla \vect{u}}^{\sfrac{1}{2}} \Lpnorm{\nabla \psi}^{\sfrac{1}{2}} \Lpnorm{\lapl \psi}^{\sfrac{1}{2}} \Lpnorm{\lapl^{2} \psi}^{}\\
        \leqslant & \frac{1}{3} \Lpnorm{\lapl^{2} \psi}^{2} + \Const \Lpnorm{\vect{u}}^{2} \Lpnorm{\nabla \vect{u}}^{2} + \Const \Lpnorm{\nabla \psi}^{2} \Lpnorm{\lapl \psi}^{2}.
    \end{align*}
   We are left to consider the last term in~\eqref{uH1}, for which we exploit
   \begin{equation*}
        \Lpnorm{\lapl \psi} \Lpnorm{\nabla \lapl \psi} \leqslant \Const \Lpnorm{\nabla \psi} \Lpnorm{\lapl^{2} \psi},
   \end{equation*}
   namely,
   \begin{align*}
        & 2\left|\duality{ \lapl f(\psi)}{\lapl^{2}\psi}\right|\\
        =& \left|2\duality{ f'(\psi) \lapl\psi}{\lapl^{2}\psi}+2\duality{ f''(\psi)|\nabla\psi|^{2}}{\lapl^{2}\psi}\right|\\
        \leqslant & \Const (1+\Lpnorm[4]{f(\psi)}) (\Lpnorm[4]{\lapl\psi}+\Lpnorm[\infty]{\nabla\psi} \Lpnorm[4]{\nabla\psi} )\Lpnorm{\lapl^{2}\psi}\\
        \leqslant & \Const (1+\Lpnorm[4]{f(\psi)}) (\Lpnorm{\lapl\psi}^{\sfrac12} \Lpnorm{\nabla\lapl\psi}^{\sfrac12} +\Lpnorm{\nabla\psi}^{\sfrac12} \Lpnorm{\nabla\lapl\psi}^{\sfrac12} \Lpnorm{\nabla\psi}^{\sfrac12} \Lpnorm{\lapl\psi}^{\sfrac12})\Lpnorm{\lapl^{2}\psi}\\
        \leqslant & \Const (1+\Lpnorm[4]{f(\psi)}) (\Lpnorm{\nabla\psi}^{\sfrac12} \Lpnorm{\lapl^{2}\psi}^{\sfrac12} + \Lpnorm{\nabla\psi}^{\sfrac32} \Lpnorm{\lapl^{2}\psi}^{\sfrac12})\Lpnorm{\lapl^{2}\psi}\\
        \leqslant & \frac 13 \Lpnorm{\lapl^{2}\psi}^{2} + \Const (1+\Lpnorm[4]{f(\psi)}^4) (1+\Lpnorm{\nabla\psi}^6).
    \end{align*}
    We finally deduce the differential inequality
    \begin{align}\label{finaldiffine}
        & \timeder \left( \Lpnorm{\nabla \vect{u}}^{2} + \Lpnorm{\lapl \psi}^{2} \right) + \nu \Lpnorm{A \vect{u}}^{2} + \Lpnorm{\lapl^{2} \psi}^{2}\\
        \leqslant & \Const  \Lpnorm{\vect{u}}^{2} \Lpnorm{\nabla \vect{u}}^{4} \nonumber \\
        & \quad {} + \Const \left(\Lpnorm{\vect{u}}^{2}  \Lpnorm{\nabla \vect{u}}^{2} + (1+\Lpnorm[4]{f(\psi)}^{4})( \Lpnorm{\nabla\psi}^{6}+1) +  (\Lpnorm{\nabla\psi}^4+1)\Lpnorm{\lapl\psi}^{2} +  \Lpnorm{\vect{g}}^{2}\right).\nonumber
    \end{align}
    Introducing
    \begin{align*}
        h& = \Const \Lpnorm{\vect{u}}^{2} \Lpnorm{\nabla \vect{u}}^{2}  \\
        g& = \Const \left( \Lpnorm{\vect{u}}^{2} \Lpnorm{\nabla \vect{u}}^{2}+(1+\Lpnorm[4]{f(\psi)}^{4})( \Lpnorm{\nabla\psi}^{6}+1) + (\Lpnorm{\nabla\psi}^4+1)\Lpnorm{\lapl\psi}^{2} + \Lpnorm{\vect{g}}^{2} \right),
    \end{align*}
    the above differential inequality reads as
    \begin{equation}\label{finaldiffineasy}
        \timeder \norm{\myH1}{\vect{z}}^{2} \leqslant h \norm{\myH1}{\vect{z}}^{2} + g.
    \end{equation}
    Thus Gronwall's lemma gives
    \begin{equation*}
        \norm{\myH1}{\vect{z}(t)}^{2} \leqslant \left( \norm{\myH1}{\vect{z}(\tau)}^{2} + \int_{\tau}^{t} g(s) \, \mathrm{d}s \right) e^{\int_{\tau}^{t} h(s) \,  \mathrm{d}s}
    \end{equation*}
    where, by~\eqref{effe44} and Lemma~\ref{lemmabase} we have
    \begin{align*}
        \int_{\tau}^{t} h(s) \, \mathrm{d}s &\leqslant \Const \dataConst^{2}\\
        \int_{\tau}^{t} g(s) \, \mathrm{d}s &\leqslant \Const  \left( \Lpnorm{\mu(\tau)}^{2}  + \dataConst^{3} + (t-\tau) \dataConst \right) \left( \dataConst^{5} + (t-\tau) \dataConst^{4} \right) e^{\Const(\dataConst^{4} + (t-\tau) \dataConst^{3})}
    \end{align*}
    and the estimate~\eqref{E:H1_boundedness}. Moreover, integrating~\eqref{finaldiffine}, we also have
    \begin{multline*}
        \int_{\tau}^{t} \left(\Lpnorm{\lapl^2 \psi(s)}^{2}+\Lpnorm{\lapl \vect{u}(s)}^{2}\right) \, \mathrm{d}s\\
        \leqslant \Const \left( \norm{\myH1}{\vect{z}_{0}}^{2} + \Lpnorm{\mu(\tau)}^{2} + \dataConst^{3} + (t-\tau) \dataConst \right)
        \left( \dataConst^{7} + (t-\tau) \dataConst^{6} \right) e^{\Const \left(\dataConst^{4} + (t-\tau)\dataConst^{3}\right)}.\qedhere
    \end{multline*}
\end{proof}

In the case of regular initial data, i.e.\ $\vect{z}_{0} \in \myH1$, we thus have the sought higher regularity for solutions:
\begin{align*}
    \vect{u}    &\in \Bochner{\Lp[\infty]}{\tau}{T}{\Hsvect_{0, \divfree}(\Omega)} \cap \Bochner{\Lp}{\tau}{T}{\Hsvect[2]_{0, \divfree}(\Omega)}\\
    \psi        &\in \Bochner{\Lp[\infty]}{\tau}{T}{\Hs[2](\Omega)} \cap \Bochner{\Lp}{\tau}{T}{\Hs[4](\Omega)}\\
    \dt{\psi}   &\in \Bochner{\Lp}{\tau}{T}{\Lp(\Omega)}
\end{align*}

\begin{corollary}\label{cor:smoothing}
    Given any symbol $\vect{g}$ satisfying~\eqref{Hp:gL2loc} and~\eqref{Hp:gM2} and any $t_{0} \in \mathbb{R}$, there exists a positive constant $\Const_{\vect{g}}(t_{0})$ such that, for any bounded set $D \subset \myH0$, there exists $T = T(|D|) > 0$ depending only on $|D| \doteq \max \{ 1, \sup_{\vect{z} \in D} \norm{\myH0}{\vect{z}} \}$ such that
    \begin{equation*}
        \norm{\myH1}{\mathbf z(t)} + \Lpnorm{\mu(t)} \leqslant C_{\vect g}(t_0), \qquad t \leqslant t_{0}, \quad \tau \leqslant t-4-T(|D|),
    \end{equation*}
    where $\vect{z}(t)$ is the solution to the problem with symbol $\vect{g}$, departing at time $\tau$ from the initial datum  $\vect{z}_{0} \in D$, and $\mu(t)$ is the corresponding chemical potential. Besides, the following integral estimates hold true
    \begin{equation*}
        \int_{t-1}^t  \Lpnorm[4]{f(\psi(s)}^{8} \, {\rm d} s \leqslant Q(M_{\vect g}(t_0))
    \end{equation*}
    for $t \leqslant t_0$, $\tau \leqslant t-4-T(|D|)$, and, for $t \leqslant t_0$, $\tau \leqslant t-5-T(|D|)$,
    \begin{equation*}
        \int_{t-1}^t  \left( \Lpnorm{\lapl \vect{u}(s)}^{2} + \Lpnorm{\lapl^{2} \psi(s)}^{2} \right) \, {\rm d}s \leqslant Q(M_{\vect g}(t_0)),
    \end{equation*}
    for some nonnegative increasing function $Q$ depending on $p$ only through a multiplicative constant.
\end{corollary}

\begin{proof}
    In order to prove the claim, we divide our argument in several steps. At each step, thanks to the estimates of the previous sections, we will be able to improve the regularity of the solution of~\eqref{E:PDE}-\eqref{E:BC} by letting the system evolve for a time interval sufficiently large (with the only exception of the initial step, however, all time steps will be taken equal to $1$).

    Firstly, thanks to Corollary~\ref{cor:diss}, for any $t_0\in \mathbb R$, any symbol $\vect{g}$ as above and any bounded set $D \subset \myH0$, there exists $T = T(|D|) > 0$ such that
    \begin{equation*}
        \norm{\myH0}{\vect{z}(t)}^{2} + 2 \Lpnorm[1]{F(\psi(t))} \leqslant 1 + \Const (1 + M_{\vect{g}}(t_{0})) \leqslant \Const (1 + M_{\vect{g}}(t_{0})), \quad t \leqslant t_{0}, \quad \tau \leqslant t-T,
    \end{equation*}
    for any $\mathbf z_0\in D$. Moreover integrating~\eqref{E:basic_energy_estimate} and \eqref{deltainq}, for $t \leqslant t_0$
    \begin{equation*}
        \int_{t-1}^t \left(\|\mathbf z(s)\|^2_{\mathcal H_1}+\Lpnorm{\nabla\mu(s)}^2\right) \, {\rm d}s \leqslant \Const (1+M_{\vect{g}}(t_{0}))+\|\mathbf z(t-1)\|^2_{\mathcal H_0}+2|F(\psi(t-1))|_1,\quad \tau\leqslant t-1,
    \end{equation*}
    so that, provided that $\tau\leqslant t-1-T$, we deduce
    \begin{equation}\label{intH1t1}
        \int_{t-1}^t \left(\|\mathbf z(s)\|^2_{\mathcal H_1}+\Lpnorm{\nabla\mu(s)}^2\right) \, {\rm d}s \leqslant \Const (1+M_{\vect{g}}(t_{0}))
    \end{equation}
    as well as
    \begin{equation}\label{E:bounddataconst}
        A_{t,s}\leqslant \Const (1+M_{\vect{g}}(t_{0})), \qquad \tau+T\leqslant t-1 \leqslant s \leqslant t\leqslant t_0.
    \end{equation}
    It thus follows from Lemmata~\ref{lemmabase}, \ref{lemma0} and~\ref{lemmaint0}
    \begin{equation*}
        \int_{t-1}^t \left( \Lpnorm{\lapl \psi(s)}^4+\Lpnorm[p+3]{\lapl \psi(s)}^2+ \Lpnorm{\mu(s)}^2\right) {\rm d}s \leqslant \Const A^2_{t,t-1}\leqslant \Const (1+M^2_{\vect g}(t_0)),\quad \tau\leqslant t-1-T,
    \end{equation*}
    allowing to prove that the functions $h$ and $g$ in the differential inequality \eqref{E:muder}, for $\tau+T<t-1<t\leqslant t_0$, satisfy
    \begin{align*}
        \int_{t-1}^t h(s) {\rm d}s\leqslant \Const A_{t,t-1}^4\leqslant \Const (1+M^4_{\vect g}(t_0)), \qquad \int_{t-1}^t g(s) {\rm d}s\leqslant \Const A_{t,t-1}^3\leqslant \Const (1+M^3_{\vect g}(t_0)).
    \end{align*}
    Hence, by the Uniform Gronwall's lemma and \eqref{E:bounddataconst}, it follows
    \begin{equation}\label{E:mudiss}
        \Lpnorm{\mu(t)}^2 \leqslant \Const \left (1+M^3_{\vect g}(t_0)\right)e^{\Const \left (1+M^4_{\vect g}(t_0)\right) }, \qquad \tau+T+2\leqslant t\leqslant t_0
    \end{equation}
    as well as
    \begin{equation*}
        \int_{t-1}^t \Lpnorm{\lapl\mu(s)}^2 {\rm d}s \leqslant \Const \left (1+M^7_{\vect g}(t_0)\right)e^{\Const \left (1+M^4_{\vect g}(t_0)\right) },\qquad \tau+T+3\leqslant t\leqslant t_0.
    \end{equation*}
    Then the first claimed integral estimate follows from~\eqref{E:intfl8} and~\eqref{E:mudiss}, while the functions $h$ and $g$ in~\eqref{finaldiffineasy} satisfy
    \begin{align*}
       \int_{t-1}^{t} h(s) \, \mathrm{d}s \leqslant \Const \left (1+M^2_{\vect g}(t_0)\right)\qquad \int_{t-1}^{t} g(s) \, \mathrm{d}s \leqslant \Const \left (1+M^8_{\vect g}(t_0)\right) e^{\Const \left (1+M^4_{\vect g}(t_0)\right) },
    \end{align*}
    provided that $ \tau+T+3\leqslant t\leqslant t_0$. Therefore, applying the Uniform Gronwall's lemma to~\eqref{finaldiffineasy} we deduce
    \begin{equation*}
        \norm{\myH1}{\vect z(t)}^2 \leqslant \Const^2_{\vect g}(t_0), \qquad \tau+T+4\leqslant t\leqslant t_0,
    \end{equation*}
    where
    \begin{equation*}
        \Const^2_{\vect g}(t_0) \doteq \Const \left (1+M^8_{\vect g}(t_0)\right) e^{\Const \left (1+M^4_{\vect g}(t_0)\right) }.
    \end{equation*}
    Finally, provided that $\tau\leqslant t-5-T$, integrating \eqref{finaldiffine} over $(t-1,t)$ we obtain
    \begin{equation*}
        \int_{t-1}^t  \left(\Lpnorm{\lapl \vect{u}(s)}^{2} + \Lpnorm{\lapl^{2} \psi(s)}^{2}\right) {\rm d} s \leqslant \Const \left (1+M^{10}_{\vect g}(t_0)\right) e^{\Const \left (1+M^4_{\vect g}(t_0)\right) }.\qedhere
    \end{equation*}
\end{proof}

\section{Continuous dependence}\label{S:continuous dependence}

\noindent In this section we obtain continuous dependence estimates of the solutions w.r.t.\ initial data and forcing terms (see Lemmata \ref{lemma93} and
\ref{lemma94}).

In order to address the first issue and to simplify notation, throughout the section we indicate by symbols with no subscripts the difference between quantities denoted by subscripts~$1$ and~$2$, i.e.,
\begin{equation*}
    f \eqdef f_{1} - f_{2}.
\end{equation*}
From~\eqref{E:PDE}-\eqref{E:BC} we easily see that the difference between two solutions $\vect{z}_{1} = (\vect{u}_{1}, \psi_{1})$ and $\vect{z}_{2} = (\vect{u}_{2}, \psi_{2})$ satisfies the system
\begin{equation}\label{difference}
    \begin{cases}
        \pt \vect{u}+(\vect{u}\cdot \nabla)\vect{u}_1+(\vect{u}_2\cdot \nabla)\vect{u}-\nu \lapl \vect{u} = \mu_1 \nabla\psi_1-\mu_2\nabla\psi_2+\vect{g}\\
        \nabla\cdot \vect{u}=0\\
        \pt\psi+\vect{u}\cdot \nabla\psi_1+\vect{u}_2\cdot \nabla\psi=\lapl\mu\\
        \mu=-\lapl\psi+f(\psi_1)-f(\psi_2)\\
        \psi(\tau)=\psi_{01}-\psi_{02},\quad \vect{u}(\tau)=\vect{u}_{01}-\vect{u}_{02}
    \end{cases}
    \quad \text{in}\quad \Omega,
\end{equation}
with the boundary conditions
\begin{equation*}
    \vect{u} = \mathbf 0, \quad \dnu{\psi} = 0 , \quad \dnu{\mu} = 0 ,\qquad \text{on $\partial \Omega$.}
\end{equation*}

\begin{lemma}\label{lemma93}
    Let $\mathbf z_{01},\mathbf z_{02}\in \myH0$ be any pair of initial data and let $\g_1,\g_2 \in \Bochner{\Lp_{\loc}}{-\infty}{t}{\Lpvect_{\divfree}(\Omega)}$ be any pair of symbols. Then there exists a constant $\Const$ such that, if $\vect{z}_{i}(t)$, $i = 1, 2$, are the solutions of~\eqref{E:PDE}-\eqref{E:BC} with initial data $\vect{z}_{0i}$ at time $\tau$ and symbol $\vect{g}_{i}$, then the following estimates hold
    \begin{equation*}
        \norm{\myH0}{\vect{z}(t)}^{2} \leqslant \left( \norm{\myH0}{\vect{z}_{0}}^{2} + \int_{\tau}^{t} \Lpnorm{\vect{g}(s)}^{2} \, \mathrm{d}s \right) e^{\Const \left( \overline{\dataConst}^{6} + (t-\tau) \overline{\dataConst}^{5} \right)}
    \end{equation*}
    and
    \begin{align*}
        &\int_{\tau}^{t} \left( \Lpnorm{\nabla \vect{u}(s)}^{2} + \Lpnorm{\nabla \lapl \psi(s)}^{2} \right) \, \mathrm{d}s\\
        \leqslant &
        \Const \left( \norm{\myH0}{\vect{z}_{0}}^{2} + \int_{\tau}^{t} \Lpnorm{\vect{g}(s)}^{2} \, \mathrm{d}s \right)
        \left( \overline{\dataConst}^{6} + (t-\tau) \overline{\dataConst}^{5} \right)
        e^{\Const \left( \overline{\dataConst}^{6} + (t-\tau) \overline{\dataConst}^{5} \right)},
    \end{align*}
    where $\overline{\dataConst}$ is defined by $\overline{\dataConst} \eqdef {\dataConst}_{(1)} + {\dataConst}_{(2)}$ and ${\dataConst}_{(i)}$ is the quantity corresponding to the initial datum $\vect{z}_{i}$ and the forcing term $\vect{g}_{i}$, for $i=1,2$.
\end{lemma}

\begin{proof}
    Recalling that, thanks to the incompressibility assumption on the velocity fields~$\vect{u}_{i}$, for $i = 1,2$, we have
    \begin{align*}
        & \duality{\mu_{i} \nabla \psi_{i}}{\vect{u}_{i}}\\
        = &- \duality{\lapl \psi_{i} \nabla \psi_{i}}{\vect{u}_{i}} + \duality{f(\psi_{i}) \nabla \psi_{i}}{ \vect{u}_{i}} = - \duality{\lapl \psi_{i} \nabla \psi_{i}}{\vect{u}_{i}} + \duality{\nabla F(\psi_{i})}{\vect{u}_{i}}\\
        = &- \duality{\lapl \psi_{i} \nabla \psi_{i}}{\vect{u}_{i}}.
    \end{align*}
    Moreover, the product of the first equation in~\eqref{difference} by $2\vect{u}$ gives
    \begin{equation}\label{differenceu1}
        \timeder \Lpnorm{\vect{u}}^{2}+2 \nu \Lpnorm{\nabla \vect{u}}^{2}+2\duality{ (\vect{u}\cdot \nabla)\vect{u}_1}{\vect{u}} = - 2\duality{\lapl\psi \nabla\psi_1}{\vect{u}} -2\duality{ \lapl\psi_2\nabla\psi}{\vect{u}}+2\duality{\vect{g}}{\vect{u}}.
    \end{equation}
    Noticing that
    \begin{align*}
        & \duality{\lapl\mu}{\lapl\psi} \\
        =&\duality{ \lapl [-\lapl\psi+f(\psi_1)-f(\psi_2)]}{\lapl\psi}\\
        =&\Lpnorm{\nabla\lapl\psi}^{2}+\duality{ f'(\psi_1)\lapl\psi}{\lapl\psi} +\duality{ [f'(\psi_1)-f'(\psi_2)]\lapl\psi_2}{\lapl\psi}\\
        & \quad {} + \duality{ f''(\psi_1) \nabla\psi (\nabla\psi_1+\nabla\psi_2)}{\lapl\psi}+\duality{ [f''(\psi_1)-f''(\psi_2)]|\nabla\psi_2|^{2}}{\lapl\psi}\\
        \geqslant & \Lpnorm{\nabla\lapl\psi}^{2}-2\alpha \Lpnorm{\lapl\psi}^{2} +\duality{ [f'(\psi_1)-f'(\psi_2)]\lapl\psi_2}{\lapl\psi}\\
        & \quad {} +\duality{ f''(\psi_1)\nabla\psi (\nabla\psi_1+\nabla\psi_2)}{\lapl\psi}+\duality{ [f''(\psi_1)-f''(\psi_2)]|\nabla\psi_2|^{2}}{\lapl\psi},
    \end{align*}
    the product of the third equation in~\eqref{difference} by $-2 \lapl \psi$ yields
    \begin{align*}
        & \timeder \Lpnorm{\nabla\psi}^{2}+2\Lpnorm{\nabla\lapl\psi}^{2}-2\duality{ (\vect{u}\cdot\nabla)\psi_1}{\lapl\psi}-2\duality{(\vect{u}_2\cdot\nabla)\psi}{\lapl\psi}\\
        \leqslant & 4\alpha \Lpnorm{\lapl\psi}^{2} - 2\duality{ [f'(\psi_1)-f'(\psi_2)]\lapl\psi_2}{\lapl\psi}\\
        & \quad -2\duality{ f''(\psi_1)\nabla\psi (\nabla\psi_1+\nabla\psi_2)}{\lapl\psi}-2\duality{ [f''(\psi_1)-f''(\psi_2)]|\nabla\psi_2|^{2}}{\lapl\psi}.
    \end{align*}
    Adding this last inequality to~\eqref{differenceu1} we have
    \begin{align}\label{differenceumu}
        & \timeder \left(\Lpnorm{\vect{u}}^{2}+ \Lpnorm{\nabla\psi}^{2}\right) +2\nu\Lpnorm{\nabla \vect{u}}^{2}+2\Lpnorm{\nabla\lapl\psi}^{2}\\
        \leqslant & 4\alpha \Lpnorm{\lapl\psi}^{2} - 2\duality{ [f'(\psi_1)-f'(\psi_2)]\lapl\psi_2}{\lapl\psi} - 2\duality{ f''(\psi_1) \nabla\psi (\nabla\psi_1+\nabla\psi_2)}{\lapl\psi}\nonumber\\
        & \quad {} - 2\duality{ [f''(\psi_1)-f''(\psi_2)]|\nabla\psi_2|^{2}}{\lapl\psi}+2\duality{ (\vect{u}_2\cdot\nabla)\psi}{\lapl\psi} \nonumber\\
        & \quad {} - 2\duality{ (\vect{u}\cdot\nabla)\psi}{\lapl\psi_2} - 2\duality{ (\vect{u}\cdot \nabla)\vect{u}_1}{\vect{u}} + 2\duality{\vect{g}}{\vect{u}}.\nonumber
    \end{align}
    Since $\Lpnorm{\lapl\psi}^{2} \leqslant \Lpnorm{\nabla\psi}\Lpnorm{\nabla\lapl\psi}$, we immediately obtain
    \begin{equation*}
        4\alpha \Lpnorm{\lapl\psi}^{2} \leqslant 4\alpha \Lpnorm{\nabla\psi}\Lpnorm{\nabla\lapl\psi} \leqslant \frac 15\Lpnorm{\nabla\lapl\psi}^{2}+\Const \Lpnorm{\nabla\psi}^{2}.
    \end{equation*}

    Before dealing with the terms arising from the double-well potential $F$, we introduce some useful estimates for terms of the form $|f^{(k)}(\psi_{1}) - f^{(k)}(\psi_{2})|$. Indeed from the fundamental theorem of calculus and by assumptions~\eqref{Hp:polynomial_growth} and~\eqref{Hp:splitting} we have
    \begin{align*}
        &\left| f^{(k)}(\psi_{1}) - f^{(k)}(\psi_{2}) \right|\\
        = & \left| \int_{\psi_{2}}^{\psi_{1}} f^{(k+1)}(s) \, \mathrm{d}s \right| \leqslant \int_{\psi_{1}}^{\psi_{2}} \left| f^{(k+1)}(s) \right| \, \mathrm{d}s \leqslant \Const \int_{\psi_{1}}^{\psi_{2}} \left( 1 + F(s)^{\sfrac{(p+1-k)}{(p+3)}} \right) \, \mathrm{d}s\\
        \leqslant & \Const \left( 1 + F(\psi_{1})^{\sfrac{(p+1-k)}{(p+3)}} + F(\psi_{2})^{\sfrac{(p+1-k)}{(p+3)}} \right) |\psi_{1} - \psi_{2}|.
    \end{align*}
    Moreover, we can estimate suitable $\Lp[q]$ norms of this difference as follows
    \begin{align}
    \label{fksmart}
        & \Lpnorm[\frac{p+3}{p+2-k}]{f^{(k)}(\psi_{1}) - f^{(k)}(\psi_{2})}\\
        \leqslant & \Const \left( \int_{\Omega} \left( 1 + F(\psi_{1}) + F(\psi_{2}) \right)^{\sfrac{(p+1-k)}{(p+2-k)}} |\psi_{1} - \psi_{2}|^{\sfrac{(p+3)}{(p+2-k)}} \, \mathrm{d}\vect{x} \right)^{\sfrac{(p+2-k)}{(p+3)}} \nonumber\\
        \leqslant & \Const \left( \int_{\Omega} \left( 1 + F(\psi_{1}) + F(\psi_{2}) \right) \, \mathrm{d}\vect{x} \right)^{\sfrac{(p+1-k)}{(p+3)}} \left( \int_{\Omega} |\psi_{1} - \psi_{2}|^{p+3} \, \mathrm{d}\vect{x} \right)^{\sfrac{1}{(p+3)}}\nonumber\\
        \leqslant & \Const \left( 1 + \Lpnorm[1]{F(\psi_{1})}^{\sfrac{(p+1-k)}{(p+3)}} + \Lpnorm[1]{F(\psi_{2})}^{\sfrac{(p+1-k)}{(p+3)}} \right) \Lpnorm[p+3]{\psi}.
        \nonumber
    \end{align}
    In light of this bound, we can now resume the estimation of the terms in the right hand side of~\eqref{differenceumu}. We have
    \begin{align*}
        & \left| 2 \duality{\left( f'(\psi_{1})-f'(\psi_{2}) \right) \lapl \psi_{2}}{\lapl \psi} \right| \\
        \leqslant & 2 \Lpnorm[\frac{p+3}{p+1}]{f'(\psi_{1})-f'(\psi_{2})} \Lpnorm[p+3]{\lapl \psi_{2}} \Lpnorm[p+3]{\lapl \psi}\\
        \leqslant & \Const \left( 1 + \Lpnorm[1]{F(\psi_{1})}^{\sfrac{p}{(p+3)}} + \Lpnorm[1]{F(\psi_{2})}^{\sfrac{p}{(p+3)}} \right) \Lpnorm{\nabla \psi} \Lpnorm[p+3]{\lapl \psi_{2}} \Lpnorm{\nabla \lapl \psi}\\
        \leqslant & \frac{1}{5} \Lpnorm{\nabla \lapl \psi}^{2} + \Const \left( 1 + \Lpnorm[1]{F(\psi_{1})}^{\sfrac{p}{(p+3)}} + \Lpnorm[1]{F(\psi_{2})}^{\sfrac{p}{(p+3)}} \right)^{2} \Lpnorm[p+3]{\lapl \psi_{2}}^{2} \Lpnorm{\nabla \psi}^{2}.
    \end{align*}
    Analogously we also obtain
    \begin{align*}
        & \left| -2 \duality{ \left( f''(\psi_{1}) - f''(\psi_{2}) \right) |\nabla\psi_2|^{2}}{\lapl \psi} \right| \\
        \leqslant & 2 \Lpnorm[\frac{p+3}{p}]{f''(\psi_{1}) - f''(\psi_{2})}^{} \Lpnorm[p+3]{\nabla \psi_{2}}^{2} \Lpnorm[p+3]{\lapl \psi}^{}\\
        \leqslant & \Const \left( 1 + \Lpnorm[1]{F(\psi_{1})}^{\sfrac{(p-1)}{(p+3)}} + \Lpnorm[1]{F(\psi_{2})}^{\sfrac{(p-1)}{(p+3)}} \right) \Lpnorm{\nabla \psi}^{} \Lpnorm{\lapl \psi_{2}}^{2} \Lpnorm{\nabla \lapl \psi}^{}\\
        \leqslant & \frac{1}{5} \Lpnorm{\nabla \lapl \psi}^{2} + \Const \left( 1 + \Lpnorm[1]{F(\psi_1)}^{\sfrac{(p-1)}{(p+3)}} + \Lpnorm[1]{F(\psi_2)}^{\sfrac{(p-1)}{(p+3)}} \right)^{2} \Lpnorm{\lapl \psi_{2}}^{4} \Lpnorm{\nabla \psi}^{2}.
    \end{align*}
    The last term involving the potential $F$ and its derivatives can be dealt with in a similar way. Again the assumptions on $F$ provide
    \begin{align*}
        & \left| -2 \duality{f''(\psi_{1}) \nabla \psi (\nabla \psi_{1} + \nabla \psi_{2})}{\lapl \psi} \right| \\
        \leqslant & 2 \Lpnorm[\frac{p+3}{p}]{f''(\psi_{1})} \Lpnorm[p+3]{\nabla \psi} \left( \Lpnorm[p+3]{\nabla \psi_{1}} + \Lpnorm[p+3]{\nabla \psi_{2}} \right) \Lpnorm[p+3]{\lapl \psi}\\
        \leqslant & \Const \left( 1 + \Lpnorm[1]{F(\psi_{1})}^{\sfrac{p}{(p+3)}} \right) \Lpnorm{\lapl \psi} \left( \Lpnorm{\lapl \psi_{1}} + \Lpnorm{\lapl \psi_{2}} \right) \Lpnorm{\nabla \lapl \psi}\\
        \leqslant & \Const \left( 1 + \Lpnorm[1]{F(\psi_{1})}^{\sfrac{p}{(p+3)}} \right) \Lpnorm{\nabla \psi}^{\sfrac{1}{2}} \Lpnorm{\nabla \lapl \psi}^{\sfrac{1}{2}} \left( \Lpnorm{\lapl \psi_{1}}^{} + \Lpnorm{\lapl \psi_{2}}^{} \right) \Lpnorm{\nabla \lapl \psi}^{}\\
        \leqslant & \frac{1}{5} \Lpnorm{\nabla \lapl \psi}^{2} + \Const \left( 1 + \Lpnorm[1]{F(\psi_{1})}^{4} \right) \left( \Lpnorm{\lapl \psi_{1}}^{4} + \Lpnorm{\lapl \psi_{2}}^{4} \right) \Lpnorm{\nabla \psi}^{2}.
    \end{align*}
    We also bound the last four terms on the right hand side of~\eqref{differenceumu}, which arise from the linear momentum equation. Thanks to Agmon's inequality, we have $\Lpnorm[\infty]{\nabla\psi}^{2} \leqslant \Const \Lpnorm{\nabla\psi}\Lpnorm{\nabla\lapl\psi}$, and thus
    \begin{align*}
        & \left| 2\duality{ (\vect{u}_2\cdot\nabla)\psi}{\lapl\psi} - 2\duality{ (\vect{u}\cdot\nabla)\psi}{\lapl\psi_2} \right|\\
        \leqslant & 2 \Lpnorm{\vect{u}_2}\Lpnorm[\infty]{\nabla\psi} \Lpnorm{\lapl\psi} + 2\Lpnorm{\vect{u}} |\nabla\psi|_\infty
        \Lpnorm{\lapl\psi_2}\\
        \leqslant & \Const \Lpnorm{\vect{u}_2}\Lpnorm{\nabla\psi}\Lpnorm{\nabla\lapl\psi} + \Const \Lpnorm{\vect{u}}
        \Lpnorm{\nabla\psi}^{\sfrac 12}\Lpnorm{\nabla\lapl\psi}^{\sfrac12}  \Lpnorm{\lapl\psi_2}\\
        \leqslant & \frac{1}{5} \Lpnorm{\nabla\lapl\psi}^{2} + \frac{\nu}{2} \Lpnorm{\nabla \vect{u}}^{2} + \Const ( \Lpnorm{\vect{u}_2}^{2}+ \Lpnorm{\lapl\psi_2}^4) \Lpnorm{\nabla\psi}^{2}
    \end{align*}
    holds. Finally, by Ladyzhenskaja inequality, it follows
    \begin{align*}
        & \left| -2\duality{ (\vect{u}\cdot \nabla)\vect{u}_1}{\vect{u}} + 2\duality{\vect{g}}{\vect{u}} \right|\\
        \leqslant & \Const \Lpnorm[4]{\vect{u}}^{2} \Lpnorm{\nabla\vect{u}_{1}} + \Const \Lpnorm{\vect{g}} \Lpnorm{\nabla \vect{u}}\\
        \leqslant & \Const \Lpnorm{\vect{u}} \Lpnorm{\nabla \vect{u}} \Lpnorm{\nabla \vect{u}_{1}} + \Const \Lpnorm{\vect{g}} \Lpnorm{\nabla \vect{u}}\\
        \leqslant & \frac{\nu}{2} \Lpnorm{\nabla \vect{u}}^{2} + \Const \Lpnorm{\nabla \vect{u}_{1}}^{2} |\vect{u}|_2^{2}+\Const |\vect{g}|_2^{2}.
    \end{align*}
    Replacing the above estimates in~\eqref{differenceumu}, we see that $\norm{\myH0}{\vect z(t)}^{2}=\Lpnorm{\vect{u}(t)}^{2}+ \Lpnorm{\nabla\psi(t)}^{2}$ satisfies
    \begin{equation}\label{E:differenceH0}
        \timeder \norm{\myH0}{\vect{z}}^{2} + \nu \Lpnorm{\nabla \vect{u}}^{2} + \Lpnorm{\nabla \lapl \psi}^{2} \leqslant \Const \Lpnorm{\vect{g}}^{2} + h \norm{\myH0}{\vect{z}}^{2},
    \end{equation}
    where $h$ is given by
    \begin{align*}
        h \eqdef & \Const \left( 1 + \Lpnorm{\nabla \vect{u}_{1}}^{2} + \Lpnorm{\nabla \vect{u}_{2}}^{2} \right)\\
        &{} + \Const \left( 1 + \Lpnorm[1]{F(\psi_{1})}^{4} + \Lpnorm[1]{F(\psi_{2})}^{4} \right) \left( \Lpnorm{\lapl \psi_{1}}^{4} + \Lpnorm{\lapl \psi_{2}}^{4} + \Lpnorm[p+3]{\lapl \psi_{2}}^{2} \right).
    \end{align*}
    By the results of the previous section we deduce
    \begin{equation*}
        \int_{\tau}^{t} h(s) \, \mathrm{d}s \leqslant \Const \left(\overline{\dataConst}^{6} + (t-\tau) \overline{\dataConst}^{5} \right)
    \end{equation*}
    so that Gronwall's lemma finally gives the claimed estimates.
\end{proof}

In order to apply the abstract framework described in Section~\ref{S:theory}, we also need the following higher order continuous dependence estimate.
\begin{lemma}\label{lemma94}
    Let $\mathbf z_{01},\mathbf z_{02} \in \myH1$ be any pair of initial data so that $\mu_{0i} \eqdef f(\psi_{0i}) - \lapl \psi_{0i} \in \Lp(\Omega)$, $i = 1,2$ and let $\vect{g}_{1},\vect{g}_{2} \in \Bochner{\Lp_{\loc}}{-\infty}{t}{\Lpvect_{\divfree}(\Omega)}$ be any pair of symbols. Then there exists a constant $\Const$ such that, if $\vect{z}_{i}(t)$, $i = 1, 2$ are the solutions of~\eqref{E:PDE}-\eqref{E:BC} with initial data $\vect{z}_{0i}$ at time $\tau$ and symbol $\vect{g}_{i}$, then the following estimate holds
    \begin{equation*}
        \norm{\myH1}{\vect{z}(t)}^{2} \leqslant e^{Q(\dataConst, t-\tau) \left( 1 + \norm{\myH1}{\vect{z}_{01}}^{2} + \norm{\myH1}{\vect{z}_{02}}^{2} + \Lpnorm{\mu_{01}}^{2} + \Lpnorm{\mu_{02}}^{2} \right)} \left( \norm{\myH1}{\vect{z}_{0}}^{2} + \int_{\tau}^{t} \Lpnorm{\vect{g}(s)}^{2} \, \mathrm{d}s \right),
    \end{equation*}
    where $Q$ is a nonnegative increasing function of its arguments.
\end{lemma}

\begin{proof}
    We start by multiplying the first equation in~\eqref{difference} by $2A\vect{u}=-2\mathbb{P}\lapl \vect{u}$, getting
    \begin{align*}
        & \timeder \Lpnorm{\nabla \vect{u}}^{2} + 2\nu \Lpnorm{A\vect{u}}^{2}\\
        =& -2\duality{ (\vect{u}\cdot \nabla)\vect{u}_1}{A\vect{u}}-2\duality{(\vect{u}_2\cdot \nabla)\vect{u}}{A\vect{u}} - 2\duality{\lapl\psi_1 \nabla\psi_1}{A\vect{u}} +  2\duality{\lapl\psi_2\nabla\psi_2}{A\vect{u}} + 2\duality{\vect{g}}{A\vect{u}}.
    \end{align*}
    The product of the third equation in~\eqref{difference} by $2\lapl^{2}\psi$, after an integration by parts, provides
    \begin{align*}
        & \ddt \Lpnorm{\lapl\psi}^{2}+2\Lpnorm{\lapl^{2}\psi}^{2}\\
        = & -2\duality{ \vect{u}\cdot \nabla\psi_1}{\lapl^{2}\psi} - 2\duality{  \vect{u}_2\cdot \nabla\psi}{\lapl^{2}\psi} - 2\duality{  f'(\psi_1)\nabla \lapl\psi}{\nabla\lapl\psi}\\
        & \quad {} - 2\duality{  [f'(\psi_1)-f'(\psi_2)]\nabla \lapl \psi_2}{\nabla\lapl \psi} - 2\duality{  f''(\psi_1)\nabla\psi \lapl\psi_1}{\nabla\lapl\psi}\\
        & \quad {} - 2\duality{ [f''(\psi_1)-f''(\psi_2)]\nabla\psi_2 \lapl\psi_1}{\nabla\lapl\psi}
        - 2\duality{  f''(\psi_2)\nabla\psi_2 \lapl\psi}{\nabla\lapl\psi}\\
        & \quad {} - 4 \duality{f''(\psi_{1}) \nabla^{2}\psi_{1} \nabla \psi}{\nabla \lapl \psi} - 4 \duality{(f''(\psi_{1}) - f''(\psi_{2})) \nabla^{2} \psi_{1} \nabla \psi_{2}}{\nabla \lapl \psi}\\
        & \quad {} - 4 \duality{f''(\psi_{2}) \nabla^{2} \psi \nabla \psi_{2}}{\nabla \lapl \psi} - 2 \duality{f'''(\psi_{1}) \nabla \psi \vectnorm{\nabla \psi_{1}}^{2}}{\nabla \lapl \psi}\\
        & \quad {} - 2 \duality{f'''(\psi_{1}) \nabla \psi_{2} \nabla \psi \cdot (\nabla \psi_{1} + \nabla \psi_{2})}{\nabla \lapl \psi}\\
        & \quad -2\l [f'''(\psi_1)-f'''(\psi_2)]\nabla\psi_2 |\nabla\psi_2|^{2},\nabla\lapl\psi\r.
    \end{align*}
    Adding together the two equations, by Assumption~\eqref{Hp:splitting} we obtain
    \begin{align*}
        & \timeder \left( \| \vect{u} \|_{\Hsvect_{0, \divfree}(\Omega)}^{2} + \Lpnorm{\lapl\psi}^{2} \right) + 2\nu | A \vect{u}|^{2}_{\Lpvect_{\divfree}(\Omega)}+2\Lpnorm{\lapl^{2} \psi}^{2}\\
        \leqslant & -2 \duality{ (\vect{u}\cdot \nabla)\vect{u}_1}{A\vect{u}} - 2\duality{(\vect{u}_2\cdot \nabla)\vect{u}}{A\vect{u}}
         - 2\duality{\lapl\psi \nabla\psi_1}{A\vect{u}}-2\duality{\lapl\psi_2\nabla\psi}{A\vect{u}}\\
        & \quad {} + 2\duality{\vect{g}}{A\vect{u}} - 2\duality{ \vect{u}\cdot \nabla\psi_1}{\lapl^{2}\psi} - 2\duality{ \vect{u}_2\cdot \nabla\psi}{\lapl^{2}\psi}+4\alpha \Lpnorm{\nabla\lapl\psi}^{2}\\
        & \quad {} - 2\duality{ [f'(\psi_1)-f'(\psi_2)]\nabla \lapl \psi_2}{\nabla\lapl \psi} - 2\duality{ f''(\psi_1)\nabla\psi \lapl\psi_1}{\nabla\lapl\psi}\\
        & \quad {} - 2\duality{ [f''(\psi_1)-f''(\psi_2)]\nabla\psi_2 \lapl\psi_1}{\nabla\lapl\psi} - 2\duality{ f''(\psi_2)\nabla\psi_2 \lapl\psi}{\nabla\lapl\psi}\\
        & \quad {} - 4 \duality{f''(\psi_{1}) \nabla^{2}\psi_{1} \nabla \psi}{\nabla \lapl \psi} - 4 \duality{(f''(\psi_{1}) - f''(\psi_{2})) \nabla^{2} \psi_{1} \nabla \psi_{2}}{\nabla \lapl \psi}\\
        & \quad {} - 4 \duality{f''(\psi_{2}) \nabla^{2} \psi \nabla \psi_{2}}{\nabla \lapl \psi} - 2 \duality{f'''(\psi_{1}) \nabla \psi \vectnorm{\nabla \psi_{1}}^{2}}{\nabla \lapl \psi}\\
        & \quad {}
       - 2 \duality{f'''(\psi_{1}) \nabla \psi_{2} \nabla \psi \cdot (\nabla \psi_{1} + \nabla \psi_{2})}{\nabla \lapl \psi}\\
         & \quad -2\duality{ [f'''(\psi_1)-f'''(\psi_2)]\nabla\psi_2 |\nabla\psi_2|^{2}}{\nabla\lapl\psi}.
    \end{align*}

    We now show that all the eighteen terms on the right hand side of the last inequality can be bounded by
    \begin{equation*}
        h \left(\Lpnorm{\nabla \vect{u}}^{2} +  \Lpnorm{\lapl \psi}^{2}\right),
    \end{equation*}
    where $h$ is an integrable quantity. Standard computations for the Navier-Stokes equation lead to
    \begin{align*}
        & \left| -2\duality{ (\vect{u}\cdot \nabla)\vect{u}_1}{A\vect{u}} - 2\duality{(\vect{u}_2\cdot \nabla)\vect{u}}{A\vect{u}} \right|\\
        \leqslant & \Const \Lpnorm{\vect{u}}^{\sfrac{1}{2}} \Lpnorm{A\vect{u}}^{\sfrac{1}{2}} \Lpnorm{\nabla \vect{u}_1}^{} \Lpnorm{A\vect{u}}^{} + \Const \Lpnorm{\vect{u}_2}^{\sfrac{1}{2}} \Lpnorm{\nabla \vect{u}_2}^{\sfrac{1}{2}} \Lpnorm{\nabla \vect{u}}^{\sfrac{1}{2}} \Lpnorm{A\vect{u}}^{\sfrac{3}{2}}\\
        \leqslant & \frac{\nu}{3} \Lpnorm{A\vect{u}}^{2}+\Const \Lpnorm{\nabla \vect{u}_{1}}^{4} \Lpnorm{\vect{u}}^{2} + \Const \Lpnorm{\vect{u}_2}^{2} \Lpnorm{\nabla \vect{u}_{2}}^{2} \Lpnorm{\nabla \vect{u}}^{2}.
    \end{align*}
    Exploiting Agmon's inequality and the interpolation inequality $\Lpnorm[\infty]{\phi} \leqslant C \Lpnorm{\phi}^{2/3}\norm{\Hs[3]}{\phi}^{1/3}$, we can bound the following three terms
    \begin{align*}
        & \left| -2\duality{\lapl\psi \nabla\psi_1}{A\vect{u}}-2\duality{\lapl\psi_2\nabla\psi}{A\vect{u}}+2\duality{\vect{g}}{A\vect{u}} \right|\\
        \leqslant & 2 [\Lpnorm[\infty]{\lapl\psi} \Lpnorm{\nabla\psi_1} +\Lpnorm{\lapl\psi_2}\Lpnorm[\infty]{\nabla\psi} +\Lpnorm{\vect{g}}]\Lpnorm{A\vect{u}}\\
        \leqslant & 2 [\Lpnorm{\lapl\psi}^{\sfrac12} \Lpnorm{\lapl^{2}\psi}^{\sfrac12} \Lpnorm{\nabla\psi_1}^{} + \Lpnorm{\lapl\psi_2}^{} \Lpnorm{\nabla\psi}^{\sfrac{2}{3}} \Lpnorm{\lapl^{2}\psi}^{\sfrac{1}{3}} + \Lpnorm{\vect{g}}^{}]\Lpnorm{A\vect{u}}^{}\\
        \leqslant & \frac{\nu}{3} \Lpnorm{A\vect{u}}^{2} + \frac{1}{7} \Lpnorm{\lapl^{2}\psi}^{2}+\Const \Lpnorm{\nabla\psi_1}^4 \Lpnorm{\lapl\psi}^{2}+\Const \Lpnorm{\lapl\psi_2}^3 \Lpnorm{\nabla\psi}^{2}+\Const \Lpnorm{\vect{g}}^{2}
    \end{align*}
    as well as the next two
    \begin{align*}
        & \left| -2\duality{ \vect{u}\cdot \nabla\psi_1}{\lapl^{2}\psi}-2\duality{ \vect{u}_2\cdot \nabla\psi}{\lapl^{2}\psi} \right|\\
        \leqslant & 2 \Lpnorm[\infty]{\vect{u}} \Lpnorm{\nabla\psi_{1}} \Lpnorm{\lapl^{2}\psi} + 2 \Lpnorm{\vect{u}_2} \Lpnorm[\infty]{\nabla\psi} \Lpnorm{\lapl^{2}\psi}\\
        \leqslant & \Const  \Lpnorm{\vect{u}}^{\sfrac{1}{2}} \Lpnorm{A \vect{u}}^{\sfrac{1}{2}} \Lpnorm{\nabla \psi_{1}}^{} \Lpnorm{\lapl^{2} \psi}^{} + \Const \Lpnorm{\vect{u}_{2}} \Lpnorm{\nabla\psi}^{\sfrac{2}{3}}\Lpnorm{\lapl^{2}\psi}^{\sfrac{4}{3}}\\
        \leqslant & \frac{\nu}{3} \Lpnorm{A\vect{u}}^{2}+ \frac{1}{7} \Lpnorm{\lapl^{2}\psi}^{2} + \Const \Lpnorm{\nabla\psi_1}^{4} \Lpnorm{\vect{u}}^{2} + \Const \Lpnorm{\vect{u}_{2}}^{3} \Lpnorm{\nabla\psi}^{2}.
    \end{align*}
    The terms arising from the double well potential can be treated using similar techniques. By interpolation, we have
    \begin{equation*}
        4\alpha |\nabla\lapl\psi|^{2}_{2} \leqslant \frac{1}{7} \Lpnorm{\lapl^{2}\psi}^{2}+\Const \Lpnorm{\lapl\psi}^{2},
    \end{equation*}
    while, by ~\eqref{fksmart}, we obtain
    \begin{align*}
        & \left| -2\duality{ [f'(\psi_1)-f'(\psi_2)]\nabla\lapl \psi_2}{\nabla\lapl\psi} \right|\\
        \leqslant & 2 \Lpnorm[\frac{p+3}{p+1}]{f'(\psi_1)-f'(\psi_2)}  \Lpnorm[p+3]{\nabla \lapl \psi_{2}} \Lpnorm[p+3]{\nabla \lapl \psi}\\
        \leqslant & \Const \left( 1 + \Lpnorm[1]{F(\psi_1)}^{\sfrac{p}{(p+3)}}+\Lpnorm[1]{F(\psi_2)}^{\sfrac{p}{(p+3)}} \right) \Lpnorm{\nabla \psi}^{} \Lpnorm{\lapl^{2} \psi_{2}}^{} \Lpnorm{\lapl^{2} \psi}^{}\\
        \leqslant & \frac{1}{7} \Lpnorm{\lapl^{2} \psi}^{2}+ \Const \left( 1 + \Lpnorm[1]{F(\psi_{1})}^{} + \Lpnorm[1]{F(\psi_{2})}^{} \right)^{2} \Lpnorm{\lapl^{2} \psi_{2}}^{2} \Lpnorm{\nabla \psi}^{2}.
    \end{align*}
    Using also Korn's inequality, from Assumptions~\eqref{Hp:growth} and~\eqref{Hp:coercivity} we deduce
    \begin{align*}
        & \left| -2 \duality{f''(\psi_1) \nabla \psi \lapl \psi_{1}}{\nabla \lapl \psi} - 4 \duality{f''(\psi_{1}) \nabla^{2} \psi_{1} \nabla \psi}{\nabla \lapl \psi} \right|\\
        \leqslant & \Const \Lpnorm[\frac{p+3}{p}]{f''(\psi_{1})} \Lpnorm[p+3]{\nabla \psi} \Lpnorm[p+3]{\lapl \psi_{1}} \Lpnorm[p+3]{\nabla \lapl \psi}\\
        \leqslant & \Const (1+ \Lpnorm[1]{F(\psi_{1})}^{\sfrac{p}{(p+3)}}) \Lpnorm{\lapl \psi} \Lpnorm[p+3]{\lapl \psi_{1}} \Lpnorm{\lapl^{2} \psi}\\
        \leqslant & \frac{1}{7} \Lpnorm{\lapl^{2} \psi}^{} + \Const \left( 1 + \Lpnorm[1]{F(\psi_{1})}^{} \right)^{2} \Lpnorm[p+3]{\lapl \psi_{1}}^{2} \Lpnorm{\lapl \psi}^{2}.
    \end{align*}
    Arguing as in the proof of Lemma~\ref{lemma93}, namely, exploiting ~\eqref{fksmart} for the first two terms, and Assumptions~\eqref{Hp:growth} and~\eqref{Hp:coercivity} for the second two, as well as Korn's inequality again, we obtain
    \begin{align*}
        & \left| -2 \duality{ [f''(\psi_1)-f''(\psi_2)]\nabla\psi_2 \lapl\psi_1}{\nabla\lapl\psi} - 4 \duality{ \left( f''(\psi_{1}) - f''(\psi_{2}) \right) \nabla^{2} \psi_{1} \nabla \psi_{2}}{\nabla \lapl \psi} \right.\\
        &\quad \left. {} - 2 \duality{ f''(\psi_2)\nabla\psi_2 \lapl\psi}{\nabla\lapl\psi} - 4 \duality{f''(\psi_{2}) \nabla^{2} \psi \nabla \psi_{2}}{\nabla \lapl \psi}\right|\\
        \leqslant & \Const \left( 1 + \Lpnorm[1]{F(\psi_{1})}^{\sfrac{(p-1)}{(p+3)}} + \Lpnorm[1]{F(\psi_{2})}^{\sfrac{(p-1)}{(p+3)}} \right) \Lpnorm{\lapl \psi_{2}}^{} \Lpnorm[p+3]{\lapl \psi_{1}}^{} \Lpnorm{\nabla \psi}^{} \Lpnorm{\lapl^{2} \psi}^{}\\
        & \quad {} + \Const \left( 1 + \Lpnorm[1]{F(\psi_{2})}^{\sfrac{p}{(p+3)}} \right) \Lpnorm{\lapl \psi_{2}}^{} \Lpnorm{\nabla \lapl \psi}^{} \Lpnorm{\lapl^{2} \psi}^{}\\
        \leqslant & \frac{1}{7} \Lpnorm{\lapl^{2} \psi}^{2} + \Const \left( 1 + \Lpnorm[1]{F(\psi_{1})}^{2} + \Lpnorm[1]{F(\psi_{2})}^{2} \right) \Lpnorm{\lapl \psi_{2}}^{2}  \Lpnorm[p+3]{\lapl \psi_{1}}^{2} \Lpnorm{\nabla \psi}^{2}\\
        & \quad {} + \Const \left( 1 + \Lpnorm[1]{F(\psi_{2})}^{4} \right) \Lpnorm{\lapl \psi_{2}}^{4} \Lpnorm{\lapl \psi}^{2}.
    \end{align*}
    We are left to consider
    \begin{align}\label{E:continuous_dependence_last_term}
        & \left| -2 \duality{ f'''(\psi_1)\nabla\psi |\nabla\psi_1|^{2}}{ \nabla\lapl\psi} - 2 \duality{f'''(\psi_{1}) \nabla \psi_{2} \, \nabla \psi \cdot \left(\nabla \psi_{1} + \nabla \psi_{2} \right)}{\nabla \lapl \psi} \right.\\
        & \quad \left. {} - 2\duality{ [f'''(\psi_1)-f'''(\psi_2)]\nabla\psi_2 |\nabla\psi_2|^{2}}{\nabla\lapl\psi} \right| \notag\\
        \leqslant & \Const ( 1 + \Lpnorm[1]{F(\psi_1)}^{\sfrac{(p-1)}{(p+3)}}) (\Lpnorm{\lapl\psi_1}^{2} + \Lpnorm{\lapl\psi_{1}}^{} \Lpnorm{|\lapl\psi_{2}}^{} +\Lpnorm{\lapl\psi_{2}}^{2}) \Lpnorm{\lapl \psi}^{} \Lpnorm{\lapl^{2}\psi}^{} \notag\\
        & \quad {} + \Const \left( 1 + \Lpnorm[1]{F(\psi_{1})}^{\sfrac{(p-2)}{(p+3)}} + \Lpnorm[1]{F(\psi_{2})}^{\sfrac{(p-2)}{(p+3)}} \right) \Lpnorm{\lapl\psi_2}^3 \Lpnorm{\nabla\psi}
        \Lpnorm{\lapl^{2}\psi} \notag\\
        \leqslant & \frac{1}{7} \Lpnorm{\lapl^{2}\psi}^{2} + \Const (1+\Lpnorm[1]{F(\psi_1)}^{2}) (\Lpnorm{\lapl\psi_1}^4  +\Lpnorm{\lapl\psi_2}^4 ) \Lpnorm{\lapl\psi}^{2}\notag\\
        & \quad {} + \Const \left( 1 + \Lpnorm[1]{F(\psi_1)}^{2} + \Lpnorm[1]{F(\psi_2)}^{2} \right) \Lpnorm{\lapl\psi_2}^6 \Lpnorm{\nabla\psi}^{2}. \notag
    \end{align}

    \begin{remark}\label{R:continuous_dependence_p=1}
        In the case $p \in [1,2)$, under the assumption $f^{(iv)}(y)$ bounded for $y \in \mathbb{R}$ we can still derive the estimate for the term
        \begin{equation*}
            \duality{|f'''(\psi_{1}) - f'''(\psi_{2})| \nabla \psi_{2} |\nabla \psi_{2}|^{2}}{\nabla \lapl \psi} \leqslant \Const \Lpnorm{\nabla \psi}^{} \Lpnorm{\lapl \psi_{2}}^{3} \Lpnorm[p+3]{\nabla \lapl \psi}^{},
        \end{equation*}
        which gives the same result as above.
    \end{remark}

    From the above inequalities, collecting terms we obtain
    \begin{align*}
        & \timeder \left( \Lpnorm{\nabla \vect{u}}^{2} + \Lpnorm{\lapl\psi}^{2} \right) + \nu \Lpnorm{A \vect{u}}^{2} + \Lpnorm{\lapl^{2} \psi}^{2}\\
        \leqslant & \Const \Lpnorm{\vect{g}}^{2} + \Const \left( \Lpnorm{\nabla \psi_{1}}^{4} + \Lpnorm{\nabla \vect{u}_{1}}^{4} + \Lpnorm{\vect{u}_{2}}^{2} \Lpnorm{\nabla \vect{u}_{2}}^{2} \right)\Lpnorm{\nabla \vect{u}}^{2} \\
        & \quad {} + \Const \left( \Lpnorm{\vect{u}_{2}}^{3} + \Lpnorm{\lapl \psi_{2}}^{3} + \left( 1 + \Lpnorm[1]{F(\psi_{1})}^{2} + \Lpnorm[1]{F(\psi_{2})}^{2} \right) \left( \Lpnorm{\lapl^{2} \psi_{2}}^{2} + \Lpnorm{\lapl \psi_{2}}^{6} + \Lpnorm{\lapl \psi_{2}}^{2}\Lpnorm[p+3]{\lapl \psi_{1}}^{2} \right) \right) \Lpnorm{\nabla \psi}^{2}\\
        & \quad {} + \Const \big( 1 + \Lpnorm{\nabla \psi_{1}}^{4}+\left( 1 + \Lpnorm[1]{F(\psi_{1})}^{2}\right) \left( \Lpnorm{\lapl \psi_{1}}^{4} + \Lpnorm{\lapl \psi_{2}}^{4} + \Lpnorm[p+3]{\lapl \psi_{1}}^{2} \right) + \left( 1 + \Lpnorm[1]{F(\psi_{2})}^{4} \right) \Lpnorm{\lapl \psi_{2}}^{4} \big) \Lpnorm{\lapl \psi}^{2}.
    \end{align*}
    Denoting by
    \begin{align*}
        h& =\Const \left( 1+\Lpnorm{\nabla \psi_{1}}^{4}+\Lpnorm{\vect{u}_{2}}^{3}+ \Lpnorm{\nabla \vect{u}_{1}}^{4} + \Lpnorm{\vect{u}_{2}}^{2} \Lpnorm{\nabla \vect{u}_{2}}^{2}  + \Lpnorm{\lapl \psi_{2}}^{3}\right.\\
        & \qquad    \left.+ \left( 1 + \Lpnorm[1]{F(\psi_{1})}^{2} + \Lpnorm[1]{F(\psi_{2})}^{2} \right) \left( \Lpnorm{\lapl^{2} \psi_{2}}^{2} + \Lpnorm{\lapl \psi_{2}}^{6} + \Lpnorm{\lapl \psi_{2}}^{2}\Lpnorm[p+3]{\lapl \psi_{1}}^{2} \right) \right.\\
        & \qquad {} \left.+\left( 1 + \Lpnorm[1]{F(\psi_{1})}^{2}\right) \left( \Lpnorm{\lapl \psi_{1}}^{4} + \Lpnorm{\lapl \psi_{2}}^{4} + \Lpnorm[p+3]{\lapl \psi_{1}}^{2} \right) + \left( 1 + \Lpnorm[1]{F(\psi_{2})}^{4} \right) \Lpnorm{\lapl \psi_{2}}^{4} \right) ,
    \end{align*}
    the above differential inequality reads as
    \begin{equation}\label{E:dipalta}
        \timeder \norm{\myH1}{\vect{z}(t)}^{2}\leqslant h(t)\norm{\myH1}{\vect{z}(t)}^{2}+\Const \Lpnorm{\vect{g}(t)}^{2}
    \end{equation}
    and depends on $p$ only through the constants $\Const$ included in the definition of $h$. Therefore, the estimate obtained by Gronwall's lemma below has the structure
    \begin{equation*}
        \norm{\myH1}{\vect{z}(t)}^{2} \leqslant e^{Q(\overline{\dataConst}, t-\tau) \left( 1 + \norm{\myH1}{\vect{z}_{01}}^{2} + \norm{\myH1}{\vect{z}_{02}}^{2} +\Lpnorm{\mu_{01}}^2 + \Lpnorm{\mu_{02}}^2 \right)} \left( \norm{\myH1}{\vect{z}_{0}}^{2} + \int_{\tau}^{t} \Lpnorm{\vect{g}(s)}^{2} \, \mathrm{d}s \right)
    \end{equation*}
    where, for any potential $F$ satisfying the assumptions~\eqref{Hp:regularity}--\eqref{Hp:polynomial_growth}, the function $Q$ depends on $\overline{\dataConst}$ only through some of its powers, which, in particular, do not depend on the growth exponent $p$ (i.e.\ the shape) of $F$.
\end{proof}

\section{Time regularity}\label{S:time regularity}

In this section we evaluate the distance in $\myH0$ between the solution and the initial datum in terms of the time span (Lemma \ref{lemma92}), and we show a smoothing property for difference of solutions (Lemma \ref{lemma95}). This will be crucial to show that Assumptions \eqref{H:continuity_forcing} and \eqref{H:time_continuity} in Theorems \ref{T:continuous_time_attractor} and \ref{T:process_attractor} hold true for system~\eqref{E:PDE}.

\begin{lemma}\label{lemma92}
    Given any symbol $\vect g$ satisfying \eqref{Hp:gL2loc} and \eqref{Hp:gM2}, there exists a positive constant $\Const$ such that the solution $\vect z(t)$, departing at time $\tau$ from an arbitrary initial datum $\vect z_0\in\myH1$, so that $\mu_{0} \in \Lp(\Omega)$ satisfies
    \begin{align*}
        & \norm{\myH0}{\vect{z}(t)-\vect{z}_0}^2\\
        \leqslant & \sqrt{t-\tau} \,\, \left( \norm{\myH1}{\vect z_{0}}^{2} + \Lpnorm{\mu_{0}}^{2} + \dataConst^{3} + (t-\tau) \dataConst \right) \left( \dataConst^{8} + (t-\tau) \dataConst^{7} \right) e^{\Const \left( \dataConst^{4} + (t-\tau) \dataConst^{3} \right)},
    \end{align*}
    for $\tau \leqslant t \leqslant t_{0}$.
\end{lemma}

\begin{proof}
    The different features of the Navier-Stokes and the Cahn-Hilliard equations force to handle separately the two variables. We preliminarily observe that, denoting the solution and the initial datum as $\vect z(t)=(\vect u(t),\psi(t))$ and $\vect z_0=(\vect u_0,\psi_0)$, respectively,
    \begin{align*}
        \Lpnorm{\vect u(t)-\vect u_0} &\leqslant \int_\tau^t \Lpnorm{\pt \vect u (s)} \, \mathrm{d}s \leqslant \sqrt{t-\tau}\,\, \| \pt \vect u \|_{L^2(\tau,t;\Lpvect_{\divfree}(\Omega))},
    \end{align*}
    meaning that we only need to properly bound the last norm. The product of the first equation in~\eqref{E:PDE} by $2\pt \u$ gives
    \begin{equation}\label{ptu}
        \nu \timeder \Lpnorm{\nabla \vect u}^2+2\Lpnorm{\pt \vect u}^2 =
        -2\duality{\vect u\cdot\nabla)\vect u}{\pt \vect u}+2\duality{ \mu\nabla\psi}{\pt\vect u}
        +2\duality{\vect g}{\pt \vect u}.
    \end{equation}
    Here, having observed that
    \begin{equation*}
        \left| 2 \duality{\mu \nabla\psi}{\dt{\vect{u}}} \right| = \left| 2 \duality{\psi \nabla \mu}{\dt{\vect{u}}} \right| \leqslant 2 \Lpnorm{\nabla \mu} \Lpnorm[\infty]{ \psi} \Lpnorm{\dt{\vect{u}}} \leqslant \Const \Lpnorm{\nabla \mu}^{} \Lpnorm{\psi}^{\sfrac{1}{2}} \Lpnorm{\lapl \psi}^{\sfrac{1}{2}} \Lpnorm{\dt{\vect{u}}}^{},
    \end{equation*}
    the right hand side can be controlled as
    \begin{align*}
        & \left| -2\duality{(\vect{u}\cdot\nabla) \vect {u}}{\dt{\vect{u}}} + 2 \duality{\mu \nabla\psi}{\dt{\vect{u}}}^{} + 2 \duality{\vect {g}}{\dt{\vect{u}}} \right|\\
        \leqslant & \Lpnorm{\dt{\vect u}}^{2} + \Const \Lpnorm{\vect u}^{} \Lpnorm{\nabla \vect{u}}^{2} |\vect u|_{\Hsvect[2]_{0, \divfree}(\Omega)} + \Const \Lpnorm{\nabla \mu}^{2} \Lpnorm{\psi} \Lpnorm{\lapl \psi}^{}
         +\Const\Lpnorm{\vect g}^2.
    \end{align*}
    Replacing this estimate in the differential equality above, we have
    \begin{equation*}
        \nu \timeder \Lpnorm{\nabla \vect u}^2 + \Lpnorm{\dt{\vect{u}}}^2 \leqslant \Const \left( \Lpnorm{\vect u} \Lpnorm{\nabla \vect{u}}^{2} |\vect u|_{\Hsvect[2]_{0, \divfree}(\Omega)}+
        \Lpnorm{\nabla\mu}^2 \Lpnorm{\psi} \Lpnorm{\lapl \psi} +\Lpnorm{\vect g}^2\right),
    \end{equation*}
    thus, integrating in time over $(\tau,t)$, thanks to Lemmas \ref{lemmabase} and \ref{lemma:ordine superiore}, we deduce
    \begin{align*}
        & \int_\tau^t \Lpnorm{\dt{\vect{u}(s)}}^2 \, \mathrm{d}s\\
        \leqslant & \left( \nu\Lpnorm{\nabla\vect{u}(\tau)}^2 +\Const \int_\tau^t \Lpnorm{\vect g(s)}^2 \mathrm{d}s +\Const \int_\tau^t [ \Lpnorm{\nabla\mu(s)}^{2} \norm{\Hs}{\psi(s)}^{} \norm{\Hs[2]}{\psi(s)}^{} + \Lpnorm{\vect u(s)}^{2} |\vect{u}(s)|^{2}_{\Hsvect[2]_{0, \divfree}(\Omega)} ] \, \mathbf{d}s \right)\\
        \leqslant & \Const \dataConst \left( \norm{\myH1}{\vect z(\tau)}^2+\Lpnorm{\mu(\tau)}^{2}  + \dataConst^{3} + (t-\tau) \dataConst \right) \left( \dataConst^{7} + (t-\tau) \dataConst^{6} \right) e^{\Const \left( \dataConst^{4} + (t-\tau) \dataConst^{3} \right)},
    \end{align*}
    which  provides the desired estimate. We now turn our attention to the order parameter. By interpolation, exploiting~\eqref{intptuL2} and Lemma~\ref{lemma:ordine superiore} again, we obtain
    \begin{align*}
        & \norm{\Hs(\Omega)}{\psi(t,\tau)-\psi_0}^2\\
        \leqslant & \Lpnorm{\psi(t,\tau)-\psi_0} \norm{\Hs[2](\Omega)}{\psi(t,\tau)-\psi_0} \\
        \leqslant & \Const \sqrt{t-\tau} \, \Big(\int_\tau^t \Lpnorm{\pt\psi(s,\tau)}^2\, \mathrm{d}s \Big)^{\sfrac12} \, \sup_{s\in [\tau,t]} \norm{\Hs[2](\Omega)}{\psi(s)}\\
        \leqslant & \Const \sqrt{t-\tau}\, \left( \norm{\myH1}{\vect z(\tau)}^2+\Lpnorm{\mu(\tau)}^{2}  + \dataConst^{3} + (t-\tau) \dataConst \right) \left( \dataConst^{5} + (t-\tau) \dataConst^{4} \right) e^{\Const \left( \dataConst^{4} + (t-\tau) \dataConst^{3} \right)}.\qedhere
    \end{align*}
\end{proof}

The following smoothing property is crucial to show that our problem fits in the theoretical setting of~\cite{Langa2010}, which was presented in Section~\ref{S:theory}.
\begin{lemma}\label{lemma95}
    There exists a positive function $Q(\cdot,\cdot)$, increasing in both arguments, such that, given a pair of symbols $\vect g_1,\vect g_2$ satisfying~\eqref{Hp:gL2loc} and~\eqref{Hp:gM2} and any pair of initial data $\vect z_{01},\vect z_{02}\in \myH1$ so that $\mu_{0i} \in \Lp(\Omega)$, $i=1,2$, there holds
    \begin{align*}
        & (t-\tau)\norm{\myH1}{\vect{z}(t)}^2\\
        \leqslant & \left(\norm{\myH0}{\vect z_{0}}^2+\Const (1+t-\tau) \int_\tau^t \Lpnorm{\vect{g}(s)}^2 \, {\rm d}s \right) e^{Q(\overline{\dataConst}, t-\tau) \left( 1 + \norm{\myH1}{\vect{z}_{01}}^{2} + \norm{\myH1}{\vect{z}_{02}}^{2} + \Lpnorm{\mu_{01}}^{2} + \Lpnorm{\mu_{02}}^{2} \right)}
    \end{align*}
    where $\vect z_i(t)$ stands for the solution to problem \eqref{E:PDE}-\eqref{E:BC} corresponding to symbol $\vect g_i$ originating at time $\tau$ from the initial datum $\vect z_{i0}$.
\end{lemma}

\begin{proof}
    Multiplying~\eqref{E:dipalta} by $(t-\tau)$, we obtain the differential inequality
    \begin{align*}
        & \ddt \big( (t-\tau)\norm{\myH1}{\vect z (t)}^2\big)\\
        \leqslant & \norm{\myH1}{\vect z(t)}^2+\Const(t-\tau)\Lpnorm{\vect g(t)}^{2} + h(t)(t-\tau)\norm{\myH1}{\vect z(t)}^2,
    \end{align*}
    where the function $h$ is given as in the proof of Lemma~\ref{lemma94}. By the second estimate in Lemma~\ref{lemma93}, we deduce
    \begin{equation*}
        \int_\tau^t \norm{\myH1}{\vect{z}(s)}^{2}\, {\rm d} s
        \leqslant \Const \left( \norm{\myH0}{\vect{z}_{0}}^{2} + \int_{\tau}^{t} \Lpnorm{\vect{g}(s)}^{2} \, \mathrm{d}s \right) \left( \overline{\dataConst}^{6} + (t-\tau) \overline{\dataConst}^{5} \right) e^{\Const \left( \overline{\dataConst}^{6} + (t-\tau) \overline{\dataConst}^{5} \right)},
    \end{equation*}
    while the integral of $h $ can be bounded as in Lemma~\ref{lemma94}. Thus, the Gronwall's lemma entails
    \begin{align*}
        & (t-\tau)\norm{\myH1}{\vect{z}(t)}^{2}\\
        \leqslant & \Const \left( \norm{\myH0}{\vect{z}_{0}}^{2} + (1+t-\tau)\int_{\tau}^{t} \Lpnorm{\vect{g}(s)}^{2} \, \mathrm{d}s \right) \left( \overline{\dataConst}^{6} + (t-\tau) \overline{\dataConst}^{5} \right) e^{\Const \left( \overline{\dataConst}^{6} + (t-\tau) \overline{\dataConst}^{5} \right)}\\
        & \qquad e^{Q(\overline{\dataConst}, t-\tau) \left(1+   \norm{\myH1}{\vect{z}_{01}}^{2}+\norm{\myH1}{\vect{z}_{02}}^{2} + \Lpnorm{\mu_{01}}^{2} + \Lpnorm{\mu_{02}}^{2} \right)},
    \end{align*}
    which is, the desired estimate.
\end{proof}

\section{Proof of the main results}\label{S:validation}

In  this section we show how, properly choosing the spaces and the operators, relying on the results of previous sections, we can apply Theorem~\ref{T:continuous_time_attractor}, and subsequently Theorem~\ref{T:process_attractor}, to our system so to prove Theorem~\ref{T:main_result} and Corollary~\ref{cor:process_result}.

Let $V$ and $H$ be the spaces $\myH1$ and $\myH0$ respectively. Observe that, whenever the symbol~$\vect{g}$ satisfies assumptions~\eqref{Hp:gL2loc} and~\eqref{Hp:gM2}, thanks to Theorem~\ref{T:existence} and Lemma~\ref{lemma93} the solution operator associated to system~\eqref{E:PDE}-\eqref{E:BC} is well-defined and continuous on $H$. Moreover, thanks to Corollary~\ref{cor:smoothing}, in studying the asymptotic behavior of solutions of~\eqref{E:PDE}-\eqref{E:BC} we can further restrict our attention to the bounded subset of $V$ given by
\begin{equation*}
    B \eqdef \{ \vect{z} \in V \mid \norm{\mathcal{H}_{1}}{\vect{z}} + \Lpnorm{\mu} \leqslant C_{\vect{g}}(t_0) \},
\end{equation*}
which is uniformly (w.r.t.\ the diameter of the set of initial data) absorbing for the solutions of~\eqref{E:PDE}-\eqref{E:BC}. Since the constraint $\Lpnorm{\mu} = \Lpnorm{f(\psi) - \lapl \psi} \leqslant \Const$ is closed w.r.t.\ the topology of $V$, we can further restrain our attention to the set
\begin{equation*}
    \mathcal{O}_{\delta}^{\mu}(B) \eqdef \mathcal{O}_{\delta} (B)\cap \{ \Lpnorm{\mu} \leqslant \Const_{\vect{g}}(t_0) \}
\end{equation*}
when discussing the exponential decay of solution towards an exponential pullback attractor.

Let $\tau \in \mathbb{R}$ be given and let $\vect{g}$ satisfy~\eqref{Hp:gL2loc} and~\eqref{Hp:gMq} (so that~\eqref{Hp:gM2} holds true as well for $t \leqslant t_{0}$, which is enough for our scopes), we denote by $U_{\vect{g}}(t, \tau)$ the solution operator to problem~\eqref{E:PDE}-\eqref{E:BC} at time $t \geqslant \tau$ with symbol $\vect{g}$ and initial data in $V$. Thanks to Lemmata~\ref{lemma:ordine superiore} and~\ref{lemma94} the process $U_{\vect{g}}(t, \tau) \colon V \to V$ is well-defined and continuous on $\mathcal{O}_{\delta}^{\mu}$. Therefore, if $t_{0}$ is the time appearing in assumption~\eqref{Hp:gMq}, the restricted family $\{ U_{\vect{g}}(t, \tau) \colon \quad \tau \leqslant t \leqslant t_{0} \}$ belongs to the class $\mathcal U(V, t_{0})$. Since the set $\mathcal{O}_{\delta}^{\mu}$ is uniformly absorbing for the family $\{ U_{\vect{g}}(t, \tau) \}$, fixing $\delta=1$ and the time span
\begin{equation*}
    \tau_{0} \doteq 5 + T \left(|\mathcal{O}_{1}^{\mu}(B)| \right),
\end{equation*}
we deduce by Corollary~\ref{cor:smoothing} and Lemma~\ref{lemma95} that $U_{\vect{g}}(t, t-\tau_{0})\in \mathcal{S}_{1, L}(B)$, where the constant $L$ depends increasingly on $\tau_{0},$ $\Const_{\vect{g}}(t_{0})$ and $M_{\vect{g}}(t_{0})$.

In this framework we can now verify assumptions~\eqref{H:continuity_forcing}--\eqref{H:time_continuity} for system~\eqref{E:PDE}-\eqref{E:BC}. Indeed, \eqref{H:back_continuous_dependence} is a straightforward consequence of Lemma~\ref{lemma94}. The coupling between the Navier-Stokes and the convective Cahn-Hilliard equations makes more involved the validation of assumptions \eqref{H:continuity_forcing} and \eqref{H:time_continuity}. In order to avoid further requirements over the symbols but \eqref{Hp:gMq}, we exploit Lemma \ref{lemma92} and  interpolation, thanks to a smoothing in the solution. The technical details of our argument are contained in the following lemma.
\begin{lemma}\label{lemma:q}
    Assume that $\g$ satisfies~\eqref{Hp:gL2loc} and~\eqref{Hp:gMq}. Then there exists a positive constant~$\Const$ depending on the exponent~$q$ in~\eqref{Hp:gMq} such that, for any initial datum $\vect{z}_{0} \in \mathcal{O}_{1}^{\mu}(B)$, the solution $\vect{z}(t) = (\vect{u}(t), \psi(t)) = U_{\vect{g}}(t, \tau) \vect{z}_{0}$ satisfies
    \begin{equation}\label{E:interpola}
        \norm{\Hsvect[\sfrac {(2q-2)} q]_{0,\divfree}(\Omega)}{\vect u(t)} + \norm{\Hs[3]}{\psi(t)} \leqslant \Const,
    \end{equation}
    for any $t \leqslant t_{0}-1$ and $\tau \leqslant t - 1 - \tau_{0}$.
\end{lemma}

\begin{proof}
    In this proof, we consider the two equations separately: first, as in~\cite{Langa2010}, we apply the Giga-Sohr argument (see~\cite{Giga1991}) to the equation
    \begin{equation}\label{NS4GS}
        \pt \u-\nu \mathbb{P}\Delta\u= \vect{h},
    \end{equation}
    where
    \begin{equation*}
        \vect h \doteq -\mathbb{P}\vect u\cdot \nabla \vect u+\mathbb{P}\mu\nabla \psi+\vect g.
    \end{equation*}
    By H\"older's and Gagliardo-Nirenberg's inequalities, we obtain
    \begin{equation*}
        \Lpnorm{\vect u\cdot \nabla \vect u} \leqslant \Const\Lpnorm[q]{\vect u}\Lpnorm[\sfrac {2q}{(q-2)}]{\nabla\vect u}
        \leqslant \Const \Lpnorm{\nabla\vect u}^{\sfrac{2(q-1)}{q}}\Lpnorm{\lapl \vect u}^{\sfrac 2q}.
    \end{equation*}
    Moreover, recalling that $f(\psi) \nabla \psi \in \Lpvect(\Omega)$ and $f(\psi) \nabla \psi \in \Lpvect_{\divfree}(\Omega)^{\perp}$
    \begin{equation*}
        \Lpnorm{\mathbb{P}\mu\nabla \psi} \leqslant \Lpnorm{\lapl \psi\nabla \psi} \leqslant \Const \Lpnorm{\lapl\psi}^{\sfrac{2(q-1)}{q}}\Lpnorm{\lapl^2 \psi}^{\sfrac 2q},
    \end{equation*}
    (actually this estimate is not optimal but, due to the previous estimates on the velocity field above, this does not have any influence on the final outcome) then Corollary~\ref{cor:smoothing} and assumption~\eqref{Hp:gMq} ensure
    \begin{equation*}
        \int_{t-2}^t \Lpnorm{\vect h(s)}^q {\rm d}s\leqslant Q(M_{\vect{g},q}(t_0)),\qquad \tau\leqslant t-6-T(|\mathcal O^\mu_1(B)|)=t-1-\tau_0,
    \end{equation*}
    which is identical to~\cite[Equation~(84)]{Langa2010}. From this estimate, arguing as in~\cite{Langa2010} we deduce the bound on $\vect{u}$
    \begin{equation*}
        \norm{\Hsvect[\sfrac{(2q-2)}{q}]_{0, \divfree}(\Omega)}{\vect{u}(t)} \leqslant Q(M_{\vect{g},q}(t_0)),
    \end{equation*}
    which is the first part of~\eqref{E:interpola}.

    We now turn our attention to the Cahn-Hilliard equation. The product of the third equation in~\eqref{E:PDE} by $-2\lapl \dt{\psi}$ leads to
    \begin{align*}
        & \timeder \Lpnorm {\nabla \lapl \psi}^2+2\Lpnorm{\nabla\pt \psi}^2\\
        = & -2\duality{ \u \cdot \nabla\pt \psi}{\lapl\psi} + 2\duality{ f'(\psi)\nabla\lapl\psi}{\nabla\pt \psi} + 2\duality{ f''(\psi) \nabla\psi \lapl\psi}{ \nabla\pt \psi}\\
        & \quad {} + 4\duality{ f''(\psi) \nabla\psi \nabla^2\psi}{ \nabla\pt \psi} + 2\duality{ f'''(\psi) \nabla\psi |\nabla\psi|^2}{\nabla\pt \psi}.
    \end{align*}
    By Ladyzhenskaja inequality and interpolation, the first term on the right hand side can be controlled as
    \begin{align*}
        &\left|-2\duality{\vect u\cdot \nabla \pt\psi}{\lapl \psi}\right|\\
        \leqslant & 2 \Lpnorm[4]{\vect u} \Lpnorm[4]{\lapl \psi} \Lpnorm{\nabla\pt\psi}\\
        \leqslant & \Const \Lpnorm{\vect u}^{\sfrac12} \Lpnorm{\nabla\vect u}^{\sfrac{1}{2}} \Lpnorm{\lapl \psi}^{\sfrac{3}{4}} \Lpnorm{\lapl^{2} \psi}^{\sfrac{1}{4}} \Lpnorm{\nabla \dt{\psi}}\\
        \leqslant & \frac{1}{4} \Lpnorm{\nabla \dt{\psi}}^{2} + \Lpnorm{\lapl^{2} \psi}^{2} + \Const \Lpnorm{\vect{u}}^{\sfrac{4}{3}}\Lpnorm{\nabla \vect{u}}^{\sfrac{4}{3}} \Lpnorm{\lapl \psi}^{2}
    \end{align*}
    and, similarly, the second one is
    \begin{align*}
        &\left|2\duality{ f'(\psi)\nabla\lapl \psi}{\nabla\pt \psi}\right|\\
        \leqslant & 2 \Lpnorm[4]{f'(\psi)} \Lpnorm[4]{\nabla\lapl\psi} \Lpnorm{\nabla\pt\psi}\\
        \leqslant & \Const \Lpnorm[4]{f'(\psi)}^{} \Lpnorm{\lapl\psi}^{\sfrac 14} \Lpnorm{\lapl^2\psi}^{\sfrac 34} \Lpnorm{\nabla\pt\psi}^{}\\
        \leqslant & \frac14 \Lpnorm{\nabla\pt\psi}+\Lpnorm{\lapl^2\psi}^2+\Const \Lpnorm[4]{f'(\psi)}^8 \Lpnorm{\lapl \psi}^4.
    \end{align*}
    Taking advantage also of Agmon's and Korn's inequalities, we compute
    \begin{align*}
        & \left|2\duality{ f''(\psi) \nabla\psi \lapl\psi}{ \nabla\pt \psi} + 4\duality{ f''(\psi) \nabla\psi \nabla^2\psi}{ \nabla\pt \psi}\right|\\
        \leqslant & 2 \Lpnorm[4]{f''(\psi)} \Lpnorm[\infty]{\nabla\psi} \Lpnorm[4]{\lapl\psi} \Lpnorm{\nabla\pt\psi} + 4 \Lpnorm [4]{f''(\psi)}\Lpnorm[\infty]{\nabla\psi} \Lpnorm[4]{\nabla^2\psi} \Lpnorm{\nabla\pt\psi}\\
        \leqslant & \Const \Lpnorm [4]{f''(\psi)} \Lpnorm{\nabla\psi}^{\sfrac{1}{2}} \Lpnorm{\nabla\lapl\psi}^{\sfrac{1}{2}} \Lpnorm{\lapl\psi}^{\sfrac{1}{2}} \Lpnorm{\nabla\lapl \psi}^{\sfrac{1}{2}} \Lpnorm{\nabla \pt\psi}^{}\\
        \leqslant & \Const \Lpnorm [4]{f''(\psi)} \Lpnorm{\nabla\psi}^{\sfrac{1}{2}} \Lpnorm{\lapl\psi} \Lpnorm{\lapl^2 \psi}^{\sfrac{1}{2}}\Lpnorm{\nabla\pt\psi}^{}\\
        \leqslant & \frac{1}{4} \Lpnorm{\nabla\pt\psi}^2+ \Lpnorm{\lapl^2\psi}^2 + \Const \Lpnorm [4]{f''(\psi)}^{4} \Lpnorm{\nabla\psi}^{2} \Lpnorm{\lapl\psi}^{4}.
    \end{align*}
    Finally,
    \begin{equation*}
        \left|2\duality{ f'''(\psi) \nabla\psi |\nabla\psi|^2}{\nabla\pt \psi}\right|
        \leqslant \Const \Lpnorm[4]{f'''(\psi)}^{} \Lpnorm{\Delta\psi}^{3} \Lpnorm{\nabla\pt\psi} \leqslant  \frac 14 \Lpnorm{\nabla\pt\psi}^2+ \Const \Lpnorm [4]{f'''(\psi)}^2 \Lpnorm{\lapl\psi}^6.
    \end{equation*}
    Collecting the above estimates, we obtain
    \begin{equation*}
        \timeder \Lpnorm{\nabla\lapl \psi}^2 \leqslant h,
    \end{equation*}
    where
    \begin{align*}
        h & = \Const \left(\Lpnorm{\lapl^{2} \psi}^{2} + \Lpnorm[4]{f'(\psi)}^8 \Lpnorm{\lapl \psi}^4+\Lpnorm{\vect u}^{\sfrac 43} \Lpnorm{\nabla\u}^{\sfrac 43} \Lpnorm{\lapl \psi}^2
        +\Lpnorm[4]{f''(\psi)}^{4} \Lpnorm{\nabla\psi}^{2} \Lpnorm{\lapl \psi}^{4} + \Lpnorm[4]{f'''(\psi)}^2 \Lpnorm{\lapl \psi}^6 \right).
    \end{align*}
    Having observed that by Assumption~\eqref{Hp:constants} we have
    \begin{align*}
        \Lpnorm[4]{f'(\psi)}+ \Lpnorm[4]{f''(\psi)}+ \Lpnorm[4]{f'''(\psi)} \leqslant  \Lpnorm[4]{f(\psi)},
    \end{align*}
    then Corollary~\ref{cor:smoothing} provides
    \begin{equation*}
        \int_{t-1}^t h(s) {\rm d}s \leqslant Q(M_{\vect{g}}(t_0)), \quad \tau \leqslant t-5-T(|\mathcal O^\mu_1(B)|),
    \end{equation*}
    by the Uniform Gronwall's lemma, we deduce
    \begin{equation*}
        \Lpnorm{\nabla\lapl \psi(t)}^2\leqslant Q(M_{\vect{g}}(t_0)) \quad t\geqslant \tau+5+T(|\mathcal O^\mu_1(B)|). \qedhere
    \end{equation*}
\end{proof}

We can now show that~\eqref{H:continuity_forcing} holds true. Having fixed $\vect{g}$ satisfying~\eqref{Hp:gL2loc} and~\eqref{Hp:gMq}, we denote by $\Const$ a generic positive constant depending only on $M_{\vect g,q}(t_0)$. Then we observe that by Corollary \ref{cor:smoothing} there holds
\begin{equation*}
    \sup_{\vect z_0\in \mathcal{O}_{1}^{\mu}(B)}\norm{\myH1}{U_{\vect g}(t,\tau)\vect z_0}\leqslant C
\end{equation*}
for $t\leqslant t_0$ and $\tau\leqslant t-\tau_0$. Besides, Corollary \ref{cor:diss} entails
\begin{equation}\label{databdd}
    \dataConst\leqslant C,\qquad t\leqslant t_0,\quad \tau\leqslant t-\tau_0, \quad \forall \vect z_0\in \mathcal{O}_{1}^{\mu}(B).
\end{equation}
Therefore, arguing as in Lemma \ref{lemma92}, for any initial datum $\vect z_0\in \mathcal{O}_{1}^{\mu}(B)$, we have
\begin{equation}\label{sqrts}
    \norm{\myH0}{U_{\vect g}(t,\tau)\vect z_0-U_{\vect g}(t-s,\tau )\vect z_0}\leqslant C\sqrt s,
\end{equation}
for $t\leqslant t_0-1$, $0\leqslant s\leqslant 1$ and $\tau\leqslant t-\tau_0$. To proceed in our argument, we need to consider the two variable separately: thus, for any initial datum $\vect{z}_{0} \in \mathcal{O}_{1}^{\mu}(B)$, we set $\left(\vect u(t,\tau),\psi(t,\tau)\right)=U_{\vect g}(t,\tau)\vect z_0$. Provided that $t\leqslant t_0-1$, $0\leqslant s\leqslant 1$, $\tau_0\leqslant r\leqslant 2\tau _0$, by interpolation and Lemma \ref{lemma:q}, we have
\begin{align*}
    & \norm{\Hsvect_{0, \divfree}(\Omega)}{\vect u(t,t-r)-\vect u(t-s,t-s-r)}\\
    \leqslant & \Const \norm{\Lpvect_{\divfree}(\Omega)}{\vect u(t,t-r)-\vect u(t-s,t-s-r)}^{\sfrac{(q-2)}{(2q-2)}}
        \norm{\Hsvect[\sfrac {(2q-2)} q]_{0,\divfree}(\Omega)}{\vect u(t,t-r)-\vect u(t-s,t-s-r)}^{\sfrac {q}{(2q-2)} }\\
    \leqslant & \Const \norm{\Lpvect_{\divfree}(\Omega)}{\vect u(t,t-r)-\vect u(t-s,t-s-r)}^{\sfrac{(q-2)}{(2q-2)}},
\end{align*}
as well as
\begin{align*}
    & \norm{H^2(\Omega)}{\psi (t,t-r)-\psi(t-s,t-s-r)}\\
    \leqslant & \Const \norm{H^1(\Omega)}{\psi (t,t-r)-\psi(t-s,t-s-r)}^{\sfrac12}
        \norm{H^3(\Omega)}{\psi (t,t-r)-\psi(t-s,t-s-r)}^{\sfrac12}\\
    \leqslant & \Const \norm{{H^1(\Omega)}}{\psi (t,t-r)-\psi(t-s,t-s-r)}^{\sfrac12}.
\end{align*}
Thus we are left to control a $\myH0$-norm which can be split into two parts as
\begin{align*}
    & \norm{\myH0}{U_{\vect g}(t,t-r)\vect z_0-U_{\vect g}(t-s,t-s-r)\vect z_0}\\
    \leqslant & \norm{\myH0}{U_{\vect g}(t,t-r)\vect z_0-U_{\vect g}(t-s,t-r)\vect z_0}
        + \norm{\myH0}{U_{\vect g}(t-s,t-r)\vect z_0-U_{\vect g}(t-s,t-s-r)\vect z_0}\\
    \leqslant & \Const \sqrt{s}+\norm{\myH0}{U_{\vect g}(t-s,t-r)\vect z_0-U_{\vect g}(t-s,t-s-r)\vect z_0},
\end{align*}
thanks to \eqref{sqrts}. In order to control the last term, we observe that, by~\eqref{E:dissipative0} with initial datum $U_{\vect g}(t-r, t-s-r)\vect z_0$, it follows
\begin{align*}
    & A_{t-s,t-r}\\
    = & \left( 1 + \norm{\myH0}{\vect{z}(t-r, t-s-r)}^{2} + 2|F(\psi(t-r, t-s-r))|_{1} + \int_{t-r}^{t-s} \Lpnorm{\vect g(s)}^{2} \, \mathrm{d}s \right)\\
    \leqslant & \Const \left( (1 + \norm{\myH0}{\vect{z}_{0}}^{2} + 2|F(\psi_{0})|_{1}) e^{-\Const s} + M_{\vect{g}}(t_{0}) \right)\\
    \leqslant & \Const (1+M_{\vect g}(t_0)).
\end{align*}
Exploiting Lemma~\ref{lemma93} with $\vect{z}_{1} = \vect{z}_{0}$ and $\vect {z}_{2}=U_{\vect{g}}(t-s, t-s-r) \vect{z}_{0}$ together with $\tau_{0} \leqslant r \leqslant 2 \tau_{0}$ and~\eqref{databdd}, the desired norm can be written as
\begin{align*}
    & \norm{\myH0}{U_{\vect g}(t-s,t-r)\vect z_0-U_{\vect g}(t-s,t-s-r)\vect z_0}\\
    = & \norm{\myH0}{U_{\vect g}(t-s,t-r)\vect z_0-U_{\vect g}(t-s,t-r) U_{\vect g}(t-r,t-s-r)\vect z_0}\\
    \leqslant & \Const \norm{\myH0}{\vect z_0- U_{\vect g}(t-r,t-s-r)\vect z_0} e^{\Const \left(1+M_{\vect g}(t_0)^{6}\right) } .
\end{align*}
Finally, we can use Lemma~\ref{lemma92} and \eqref{databdd} so to obtain
\begin{equation*}
    \norm{\myH0}{\vect z_0- U_{\vect g}(t-r,t-s-r)\vect z_0}
        \leqslant \Const \sqrt{s} \left( \Lpnorm{\mu(t-s-r)}^{2} + A_{t-r,t-s-r}^{3} \right) A_{t-r,t-s-r}^{8}  e^{\Const A_{t-r,t-s-r}^{4} },
\end{equation*}
where the first term on the right hand side is bounded by
\begin{equation*}
    \Lpnorm{\mu(t-r-s)}^{2} \leqslant \Const_{\vect g}^2(t_0)
\end{equation*}
as a consequence of Corollary~\ref{cor:smoothing} and of the absorbing set considered. This together with the H\"older's inequality and \eqref{databdd} yields the uniform estimate
\begin{align*}
    A_{t-r,t-s-r} \leqslant \Const (1+M_\g(t_0))
\end{align*}
for all $t \leqslant t_{0} - 1$, $0 \leqslant s \leqslant 1$ and $\tau_{0} \leqslant r \leqslant 2 \tau_{0}$. We thus obtain
\begin{equation*}
    \norm{\myH0}{\vect z_0- U_{\vect g}(t-r,t-s-r)\vect z_0} \leqslant \Const \sqrt{s},
\end{equation*}
which, replaced in the above inequalities, gives
\begin{align*}
    & \norm{\myH1}{U_{\vect{g}}(t, t-r) \vect{z}_{0} - U_{\vect{g}}(t-s, t-s-r) \vect{z}_{0}}\\
    \leqslant & \Const \norm{\myH0}{U_{\vect{g}}(t, t-r) \vect{z}_{0} - U_{\vect{g}}(t-s, t-s-r) \vect{z}_{0}}^{\sfrac{(q-2)}{2(q-1)}}\\
    \leqslant & \Const s^{\sfrac {(q-2)}{4(q-1)}} + \Const \norm{\myH0}{\vect{z}_{0} - U_{\vect{g}}(t-r, t-s-r) \vect{z}_{0}}^{\sfrac{(q-2)}{2(q-1)}}\\
    \leqslant & \Const s^{\sfrac {(q-2)}{4(q-1)}},
\end{align*}
for any $\vect{z}_{0} \in \mathcal{O}_{1}^{\mu}(B)$ and $t \leqslant t_{0} - 1$, $0 \leqslant s \leqslant 1 = \epsilon_{0}$, $\tau_{0} \leqslant r \leqslant 2\tau _{0}$, so that~\eqref{H:continuity_forcing} holds true.

We now turn our attention to~\eqref{H:time_continuity}, whose proof is now straightforward and follows from interpolation, Lemmata~\ref{lemma:q} and~\ref{lemma92}: indeed, for any $\vect{z}_{0} \in B$,
\begin{equation*}
    \norm{\myH1}{U_{\vect g}(t,t-r)\vect z_0-U_{\vect g}(t-s,t-r)\vect z_0}
    \leqslant \Const s^{\sfrac {(q-2)}{4(q-1)}},
\end{equation*}
for any $t \leqslant t_{0}$, $\tau_{0} \leqslant r \leqslant 2 \tau_{0}$, $0 \leqslant s \leqslant 1 = \epsilon_{0}$. This end the proof of Theorem~\ref{T:main_result}.

Finally, if~\eqref{Hp:gMq} holds uniformly for $t_{0} \in \mathbb{R}$ (exactly as~\eqref{Hp:gM2}), then it is easy to see that the above argument applies for all times to the process $U_{\vect{g}}(t, \tau)$. In particular, \eqref{H:forward_continuous_dependence} follows from Lemma~\ref{lemma94}. This proves Corollary~\ref{cor:process_result}.

%

\bibliographystyle{siam}
\bibliography{Pullback}{}

\end{document}